%% file: Masterfile.tex
\documentclass{amsart}
\usepackage{ifthen}
\usepackage[english]{babel}
\usepackage[T1]{fontenc} 
\usepackage[utf8]{inputenc} 
\usepackage{amssymb,amsmath,amsfonts,amsthm, mathrsfs}
\usepackage{babel}
\usepackage{ellipsis}
\usepackage{textcomp}
\usepackage[pdftex]{graphicx}
\usepackage{eso-pic}
\usepackage{tikz}
\usetikzlibrary{matrix, arrows, automata}
\usepackage{float}
\usepackage{color}

\usepackage{bibgerm}
\usepackage{enumerate} 
\usepackage[unicode]{hyperref} 
 \hypersetup{colorlinks=true,linkcolor=black,citecolor=black,pdfpagemode=UseOutlines,
 pdftitle={},pdfkeywords={},
  pdfauthor={Anja Wittmann},pdfsubject={mathematics}}

\newcommand{ \R }{ \mathbb{R} }
\newcommand{\N}{\mathbb{N}}
\newcommand{\CC}{\mathbb{C}}
\newcommand{ \n }{ \nabla }
\newcommand{ \V }{ T(M/B) }
\newcommand{ \Ho }{ T_H M }

\newcommand{\A}{\mathbb{A}}
\newcommand{\AD}{\hat{A} \left(\n^{M/B} \right)}

\newcommand{\Int}{\int\limits_{M/B}}

\newcommand{\Z}{\mathbb{Z}}

\newcommand{\PP}{\mathbb{P}_t}

\newcommand{\eps}{\varepsilon}
\newcommand{\trev}{\tr^{\ev}}
\newcommand{\del}{\partial}

\DeclareMathOperator{\im}{im}

\DeclareMathOperator{\ind}{ind}

\DeclareMathOperator{\tr}{tr}

\DeclareMathOperator{\chern}{ch}

\DeclareMathOperator{\RE}{Re}
\DeclareMathOperator{\IM}{Im}

\DeclareMathOperator{\supp}{supp}
\DeclareMathOperator{\ev}{ev}
\DeclareMathOperator{\odd}{odd}
\DeclareMathOperator{\grad}{grad}
\DeclareMathOperator{\spec}{spec}

\DeclareMathOperator{\dR}{dR}
\DeclareMathOperator{\loc}{loc}

\swapnumbers
\theoremstyle{plain}
\newtheorem{theorem}{Theorem}[section]
\newtheorem{lemma}[theorem]{Lemma}
\newtheorem{prop}[theorem]{Proposition}

\newtheorem{propdef}[theorem]{Proposition and Definition}

\theoremstyle{definition}
\newtheorem{defn}[theorem]{Definition}
\newtheorem{ass}[theorem]{Assumption}

\newtheorem{rem}[theorem]{Remark}

\theoremstyle{remark}

\newtheorem*{notation}{Notation}

\newenvironment{prf}{\begin{proof}[Proof:]}{\end{proof}}
\newenvironment{proofthm}[1]{\begin{proof}[Proof of Theorem #1:]}{\end{proof}}

\title[$\tilde{\eta}$ for Dirac operators with kernel over a hypersurface]{Analytical index and $\tilde{\eta}$-forms for Dirac operators with $1$-dimensional kernel over a hypersurface }

\author{Anja Wittmann}
\address{Mathematisches Institut, Albert-Ludwigs-Universität Freiburg, Eckerstr. 1, 79104 Freiburg, Germany}
\email{anja.wittmann@math.uni-freiburg.de}
\date{}
\begin{document}

\begin{abstract}
We generalize the transgression formula for the $\tilde{\eta}$-form of Bismut, Cheeger and Berline, Getzler, Vergne for vertical Dirac operators on a fibre bundle $\pi\colon M \rightarrow B$ with odd dimensional fibres where the Dirac operators have locally at most one eigenvalue of multiplicity one crossing zero transversally.
\end{abstract}
\maketitle
\tableofcontents

\pagenumbering{arabic}
\setcounter{page}{1}
\numberwithin{equation}{section}
\parindent 0pt

\input{Intro2}

\input{Requirements2}

\input{S1Example2}

\input{GeneralCase2}

\bibliographystyle{alphadin}
\bibliography{Literatur}

\end{document}

%% file: Intro2.tex
The Atiyah-Singer family index theorem states that a priori two very different quantities of a family of elliptic differential operators, namely the analytical and the topological or cohomological index, agree \cite[Theorem 3.1]{AtiyahSingerIV}.  We want to state its refinement on the level of differential forms, see \cite[Theorem 10.32]{BGV}. Let $X\hookrightarrow M\xrightarrow\pi B$ be a Riemannian fibre bundle of smooth manifolds with compact even dimensional fibres. For a chosen spin structure on the fibres let $V=\Sigma X \otimes L$ be a twisted fibrewise Dirac bundle with associated fibrewise Dirac operators $D$. Assume that the kernels of $D_b$, $b\in B$ have constant dimension. Then
\begin{equation}\label{ASFamily}
d\tilde{\eta}=\Int \AD \chern \left( L, \n^L\right) - \chern\left(\ker D^+ \ominus \ker D^-\right).
\end{equation}
We see that the $\tilde{\eta}$-form of Bismut and Cheeger transgresses between the cohomological and the analytical index. \newline
In the case of even dimensional fibres one can always deform the family of Dirac operators without changing the cohomology classes to obtain a family of operators with constant kernel dimension, see \cite[Theorem A1]{AtiyahKTheorie} or  \cite[Section 9.5]{BGV}. Hence it is reasonable to assume constant kernel dimension from the beginning.   
\newline
The situation for odd dimensional fibres is somehow contrary to that and varying kernel dimension is a phenomenon which cannot be omitted in general. We know that a family of Dirac operators parametrized by a manifold $B$ corresponds to a class in the topological $K$-theory $\ind D\in K^*(B)$, called the analytical index, where $*=0$ if the operators act on even dimensional manifolds and $*=1$ if they act on odd dimensional manifolds. Since constant kernel dimension implies vanishing $K^1$-class \cite[Theorem 4.1]{Ebert} we see that there can be topological obstructions to constant kernel dimension in odd dimensions. In the case where the operators are parametrized by a circle $B=S^1$ this obstruction class is given by the spectral flow in $K^1\left(S^1\right)\cong\Z$ which is a count with sign of eigenvalues crossing zero. \newline
In the present article we will consider families with odd dimensional fibres and where the fibrewise Dirac operators have one eigenvalue of multiplicity one crossing zero transversally. Therefore there is a hypersurface $B_0 \overset{i}{\hookrightarrow} B$ on which the kernels form a line bundle $\ker D \rightarrow B_0$ and the operators are invertible on $B\setminus B_0$. 
\begin{theorem}[{cf. \cite[Corollary 9.22]{BGV}}]\label{TheoremEinleitung}
Let $\A_t=\sqrt{t}D+\n^{\pi_* V}-\frac{c(T)}{4\sqrt{t}}$, $t>0$, be the Bismut superconnection. Then as $L^1_{\loc}$-currents
\begin{equation}
\lim_{t\to\infty}\tr^{\odd}\left(\exp\left(-\A_t^2\right)\right)=-\delta_{B_0}\tr\left(\exp\left(-\left(\n^{\ker}\right)^2\right)\right),
\end{equation}
where $\delta_{B_0}$ is the current of integration over the submanifold $B_0$ and $\n^{\ker}$ is the projection of $\n^{\pi_* V}$ onto $\ker D\rightarrow B_0$.
\end{theorem}
We know by \cite[Theorem 2.10]{BismutFreed} that $\tr^{\odd}\left(\exp\left(-\A_t^2\right)\right)$ is a representative for the Chern character of the analytical index for any $t>0$. Therefore we get a representative which is determined by the analytical data of the hypersurface where the Dirac operators have a kernel and the line bundle $\ker D\rightarrow B_0$. We see that the component in $H^1_{\dR}(B)$ is captured by the spectral flow as we already know by \cite[Section 7]{APS}.
We can also reformulate Theorem \ref{TheoremEinleitung} using \cite[Theorem 2.10]{BismutFreed} as   
\begin{equation}
\chern\left(\ind D\right) = i_! \chern\left(\ker D\rightarrow B_0\right)\in H^{\odd}_{\dR}(B).
\end{equation}
To understand the analytical index just by the knowledge of the eigenvalues and eigenspaces was the main motivation for \cite{DK}. In contrary to our article, R. Douglas and J. Kaminker investigated the influence of the multiplicity of the eigenvalues on the $K^1$-index.\newline
Furthermore we investigated whether the $\tilde{\eta}$-form, introduced by J.-M. Bismut and J. Cheeger in \cite{BC}, exists in this setting. It was there used to compute the adiabatic limit of $\eta$-invariants  but can also be seen as a generalization of the transgression forms introduced by D. Quillen in \cite{Quillen}. For its existence as a smooth differential form constant kernel dimension was crucial.

\begin{theorem}[{cf. \cite[Theorem 4.95]{BC}, \cite[Theorem 10.32]{BGV}}]
The $\tilde{\eta}$-form of Bismut and Cheeger
\begin{equation}
\tilde{\eta}=\frac{1}{\sqrt{\pi}}\sum_k \left(2\pi i\right)^{-k} \left[\int\limits_0^{\infty}\tr^{\ev}\left(\frac{d \A_t}{dt}\exp\left(-\A_t^2\right)\right) dt\right]_{[2k]} \in L^1_{\loc}\left(B, \Lambda^{\ev}T^*B\right)
\end{equation}
exists as differential form with locally integrable coefficients. Its differential as a current is given by
\begin{equation}\label{ddeta}
d\tilde{\eta}= \Int \AD \chern\left(L, \n^L\right) + \delta_{B_0}\chern\left(\ker D\rightarrow B_0 , \n^{\ker}\right).
\end{equation}

\end{theorem}
Just as formula (\ref{ASFamily}) this theorem gives a refinement on the level of differential forms of the Atiyah-Singer family index theorem  \cite[Theorem 3.4]{APS} in odd dimensions.  From the exact estimates in the proofs we can even see that $\tilde{\eta}$ is smooth on $B\setminus B_0$ and $i^*\tilde{\eta} \in \Omega^{\bullet}\left(B_0\right)$ is smooth, too. The only singularities are jumps at the hypersurface $B_0$. This can also be seen in the formula for its differential (\ref{ddeta}).\newline
In section \ref{example} we will consider an example of a line bundle over a sphere bundle $S^1\hookrightarrow M\xrightarrow\pi B$ there the above assumption on the eigenvalues of the fibrewise Dirac operators is fulfilled. We explicitly calculate $\tilde{\eta}$ and see that in these calculations the Bernoulli polynomials play an important role. The differential of $\tilde{\eta}$ fulfills formula (\ref{ddeta}) as expected.

\subsection*{Acknowledgements}
The content of this article is part of the PhD thesis of the author. So first and foremost the author wants to thank her advisor Sebastian Goette for his excellent support during the last years. Furthermore the author is thankful for fruitful discussions with Johannes Ebert and Cornelius Schr\"oder. She was financially supported by the DFG Graduiertenkolleg 1821 \glqq Cohomological Methods in Geometry \grqq .

%% file: Requirements2.tex
\section{Fibrations and the Bismut superconnection}\label{preliminaries}
In this chapter we will fix some notation and the situation of families of manifolds we are working with. For more details see \cite[Chapter 4]{BC} or also \cite[Chapter 9, 10]{BGV}. \newline
Let $X\hookrightarrow M \overset{\pi}{\twoheadrightarrow} B $ be an oriented Riemannian fibre bundle with closed odd dimensional fibres $X$ over an oriented, connected Riemannian manifold $\left(B, g_B\right)$. We will also assume that $B$ is compact. We stated our theorems in the introduction also for non-compact base $B$ and the convergence was in $L^1_{\loc}$. However everything in this article is local in $B$ and instead of using that $B$ is compact the estimates also hold on every compact subset $K\subset B$. Therefore we will for simplicity and without loss of generality suppose $B$ to be compact. We denote the vertical tangent bundle by $\V=\ker d\pi$ and choose a horizontal distribution $\Ho\cong\pi^* TB$ such that $TM=\V \oplus \Ho$. We will denote vertical local orthonormal frames by $e_i$ and horizontal ones by $f_{\alpha}$. We take the metric $g=g_{M/B}\oplus\pi^* g_B$ and the associated Levi-Civita-connection $\n^M$. The projected connection onto $\V$ is denoted by $\n^{M/B}$ and we define a connection $\n^{\oplus}=\n^{M/B}\oplus\pi^*\n^B$ which has torsion 
\begin{equation}\label{torsion}
 T(U,V)= \n^{\oplus}_U V - \n^{\oplus}_V U - [U,V] \in \V
 \end{equation}
for horizontal vectors $U,V \in\Ho$. \newline
For a vertical Dirac bundle $ \left(V, g^V, \n^V, c\right) $ with associated fibrewise Dirac operator 
\[ D=\sum_i c\left(e_i\right) \n^V_{e_i}\colon\Gamma\left(M, V\right)\rightarrow\Gamma\left(M,V\right)\]
we get the associated vector bundle $\pi_* V\rightarrow B$ whose infinite dimensional fibres are the fibrewise smooth sections of $V$. We will make use of the natural isomorphism $\Gamma\left(B,\pi_*V\right)\cong\Gamma\left(M,V\right)$ without actually mentioning it. The induced connection 
\begin{equation}\label{connection}
\n^{\pi_* V}=\n^V + \frac{1}{2}k ,
\end{equation}
where $k$ is the mean curvature of the fibres, is Euclidean with respect to the $L^2$-metric on $\pi_* V$. The Bismut superconnection \cite[Definition 3.2]{Bismut} is then given by
\[ \A_t= \sqrt{t} D + \n^{\pi_* V} - \frac{1}{4 \sqrt{t}} c(T)\colon \Omega^{\bullet}\left(B, \pi_* V\right)\rightarrow\Omega^{\bullet}\left(B,\pi_* V\right),\]
where we assume that $dy_{\alpha}$ and $c(e_i)$ anticommute. It follows from the transgression formula, see for example \cite[Eq. (4.38)]{BC}, that 
\begin{equation}\label{tragre}
d\int\limits_{T}^s \tr^{\ev}\left(\frac{d\A_t}{dt}\exp\left(-\A_t^2\right)\right) dt = \tr^{\odd}\left(\exp\left(-\A_T^2\right)\right) - \tr^{\odd}\left(\exp\left(-\A^2_s\right)\right).
\end{equation}
If $\V$ is spin, let $\Sigma$ denote the spinor bundle for a chosen spin structure. Then we know by \cite[Theorem 2.10]{BismutFreed} that for $V=\Sigma \otimes L$
\begin{equation}
\frac{1}{\sqrt{\pi}}\lim_{T\to 0} \left( \sum_k (2\pi i)^{-k} \tr^{\odd}\left(\exp\left(-\A_T^2\right)\right)_{[2k+1]}\right)= \Int \AD \chern\left(L, \n^L\right)
\end{equation}
which is a representative for the odd Chern class of the family $\{D_b\}_{b\in B}$. One should notice that we use Chern-Weil forms of the form $P\left(-F/2\pi i\right)$ for a curvature $F$ of a connection.

%% file: S1Example2.tex
\section{Example of a $S^1$-bundle}\label{example}

Before we come to the more general case, we will consider one special example of a family of Dirac operators. We are following the requirements in \cite{ZhangPaper}, where we adopt the construction of the fibre bundle but change the Dirac bundle. \newline
Let $\left(E, g^E\right)\xrightarrow\pi \left(B, g_B\right)$ be a real, Euclidean, oriented vector bundle of rank $2$ and denote by $\n^E$ a Euclidean connection on it. We write $T_H E\cong \pi^* TB$ for the horizontal bundle of $TE$, which is specified by $\n^E$. We define the metric $ g_{TE}=\pi^*g_E \oplus \pi^*g_B $ on $TE=\pi^*E \oplus T_H E$. Let
\begin{align*} 
M &= \{ v\in E \, | \, g_E(v,v)=1 \}, \\
T_H M &= T_H E|_M , \\
TM&= \ker d\pi \oplus T_H M = T(M/B) \oplus T_H M , \\
g&=g_{TE}|_M=g_{M/B}\oplus \pi^*g_B.
\end{align*}
 $M\xrightarrow\pi B$ is an oriented, Riemannian fibre bundle with fibres $X\cong S^1$. Let $e \in \Gamma\left(M, \V\right)$ be the unique positive oriented section of length $g_{M/B}(e,e)=1$ which trivializes $\V\cong M\times \R$. \newline
Let $\left(V, g^V, \n^V \right)\rightarrow M $ be a Hermitian line bundle with compatible connection. By setting $c(e)=-i$ we make it into a vertical Dirac bundle with Dirac operator $D=-i \n^V_e$. The fibrewise holonomies $e^{-2\pi i a}$ give rise to a smooth function $a\colon B\rightarrow \R\backslash\Z$. 

\begin{ass}\label{assumptiona}
$a\colon B \rightarrow \R\backslash \Z$ crosses $[0]$ transversally.
\end{ass}

We denote the codimension $1$ submanifold $a^{-1}([0])\subset B$ by $B_0$. We give $B_0$ the orientation such that 
\[\left(v_2,..., v_{m-1}\right) \in o_x(B_0) \Leftrightarrow \left(\grad_x a,  v_2,..., v_{m-1}\right)\in o_x(B)  .\]

\begin{rem}
If the holonomies give rise to a non-constant $a\colon B\rightarrow\R\backslash\Z$ we can always modify the connection $\n^V$ to fulfill assumption \ref{assumptiona}. Sard's Theorem makes sure that there exists an element $[x]\in \im a$ which is a regular value. The connection
\[ \tilde{\n}^V = \n^V - ixe^* \]
then gives rise to 
\[ \tilde{a}=a-[x]\colon B\rightarrow \R\backslash\Z \]
which crosses zero transversally.
\end{rem}

\begin{lemma}
The vector spaces $\ker D_b, b\in B_0$ form a smooth line bundle $\ker D \rightarrow B_0$ over the hypersurface $B_0$ and $D_b$ is invertible for $b\in B\backslash B_0$.
\end{lemma}

\begin{prf}
A straight-forward calculation shows that the eigenvalues of $D_b$ are given by $\left(k+ a(b)\right)_{k\in\Z}$. Therefore the lemma follows by assumption \ref{assumptiona} and \cite[Corollary 9.11]{BGV}.
\end{prf}

\begin{lemma}[{\cite[Lemma 1.3]{ZhangPaper}}]
Let $T$ be the torsion of $\n^{\oplus}$ as in (\ref{torsion}).Then
\begin{equation}
g(T(U,V), e)=de^*\left(U,V\right)
\end{equation}
an hence $T$ defines a two-form which we will also denote by $T\in \Omega^2(B)$. 
\end{lemma}

\begin{lemma}[{\cite[Lemma 1.6]{ZhangPaper}}]
 The mean curvature $k$ of the fibres vanishes and therefore (\ref{connection}) leads to
 \[ \n^{\pi_* V}_X \sigma = \n^V_{X^H} \sigma .\]
\end{lemma}

\begin{rem}
To facilitate the computations for the next theorem we calculate the following summands of the curvature $\A_t^2$ of the Bismut superconnection. We write $[.,.]$ for the supercommutator with respect to the grading of $\Omega^{\bullet}(B)$ and keep in mind that $dy_{\alpha}$ and $c(e_i)$ anticommute.
\begin{align*}
  [ c\left(T\right),\n^{\pi_*V}]&=0 \\
 [D, c(T)]&=2Dc(T)\\
  c(T)^2&=-T^2 .  
\end{align*}
For local considerations we choose an open subset $U\subset B$ such that there exists an eigensection $\sigma\in\Gamma\left(U, \left.\pi_* V\right|_U\right) \cong \Gamma\left(\pi^{-1}(U), \left.V\right|_{\pi^{-1}(U)}\right)$ which trivializes $\left.V\right|_{\pi^{-1}(U)}$. We denote the corresponding eigenvalue by $f\colon U \rightarrow \R$ where $f \mod \Z=a$. Since $D=-i\n^V_e$ the connection $\n^V$ locally looks like
\[ \n^V = d + ife^* + \gamma \]
for $\gamma\in\Gamma\left(U, \left. T^*_H M\right|_U \otimes_{\R} \CC\right)$. We will assume that
\[ \gamma = \pi^* \beta .\]
Then we can calculate that in this trivialization
\begin{align*} [D,\n^{\pi_* V}]&=df \\
\left(\n^{\pi_* V}\right)^2 &= d \beta + ifT-T\n^V_e.
\end{align*}

  \end{rem}

 \begin{theorem}\label{etas1}
 Set
 \begin{equation}
 \alpha (T) := \frac{1}{\sqrt{\pi}}\int\limits_0^T \tr^{\ev}\left(\frac{d\A_t}{dt} \exp\left(-\A_t^2\right)\right)dt \in\Omega^{2\bullet}(B).
 \end{equation}
 For each $b\in B$ the differential form $\alpha(T)_b$ converges as $T\to\infty$ to
\[ \hat{\eta}_b= \lim_{T\to\infty} \alpha(T)_b \in\Lambda^{2\bullet} T_b^*B\]
and we get that
 \begin{align*}
\tilde{\eta}_b &= \sum_j \frac{1}{(2\pi i)^j} \hat{\eta}_{2j} \\
&=  \exp\left(-\frac{d\beta+ifT}{2\pi i}\right) \begin{cases}
 \sum\limits_{k=1}^{\infty} \frac{B_k (a)}{k!} \left(\frac{T}{2\pi}\right)^{k-1}        ,&\text{ if } b\in B\backslash B_0      \\
\sum\limits_{k=1}^{\infty} \frac{B_{2k}}{(2k)!} \left(\frac{T}{2\pi}\right)^{2k-1}, &\text{ if } b\in B_0
\end{cases}           \\
&= \exp\left(-\frac{d\beta+ifT}{2\pi i}\right) \left(-\frac{T}{2\pi}\right)^{-1} \begin{cases} \left(\frac{T/2\pi}{\exp\left(T/2\pi\right) -1} \exp\left(\frac{aT}{2\pi}\right)-1\right), &\text{ if } b\in B\backslash B_0 \\
 \left( \frac{T/2\pi}{\exp\left(T/2\pi\right) -1} -1 +\frac{T}{4\pi} \right), &\text{ if }b \in B_0
 \end{cases}
 \end{align*}
where we see $a\in [0,1)$, $f\colon U\rightarrow \R$ describes a local eigenvalue of $D$, $\beta$ is the corresponding horizontal connection form of the Dirac bundle in this trivialization and $B_{2k}$ are the Bernoulli numbers and $B_k (a)$ the Bernoulli polynomials.
 \end{theorem}
 \begin{rem}
An easy computation shows that our formula for $\hat{\eta}$ corresponds to the one given in \cite[(5.23)]{Savale} for $r=f$. The difference lies in the fact that in our case $f$ is a function depending on the parameter $b\in B$ such that we get a differential form which has jumps, whereas in \cite{Savale} $r\in\R$ is seen as a fixed integer and $\hat{\eta}$ is seen as a smooth differential form for each $r\in\R$.
\end{rem}

\begin{rem}
We prove that the right hand side of the formula in Theorem \ref{etas1} is independent of the chosen trivialization. Therefore we take another local eigensection $\sigma_1$ with
 \[ D\sigma_1=f_1 \sigma_1 .\]
Since the eigenvalues of $D$ differ by integers, there exists a $k\in\Z$ such that $f_1=f+k$ and $\sigma_1=e^{ik\varphi}\sigma_0$. The local horizontal connection $1$-form $\beta_1$ in this trivialization is then defined by
\[  \beta_1=\frac{g^V\left(\n^V \sigma_1, \sigma_1\right)}{g^V\left(\sigma_1, \sigma_1\right)}\] 
and we can conclude that
 \begin{align*}
 \beta_1 &= d\left(e^{-ik\varphi}\right)e^{ik\varphi} + \beta \\
 &= -ik e^* + \beta.
 \end{align*}
 It follows that
 \begin{equation*}
 d\beta_1= -ikT + d\beta 
 \end{equation*}
 and therefore 
 \begin{equation*}
 \exp\left(-\frac{d\beta +ifT}{2\pi i}\right)=\exp\left(-\frac{d\beta_1 + if_1 T}{2\pi i}\right).
 \end{equation*}
\end{rem}

 \begin{proofthm}{\ref{etas1}}

\begin{align*}
\hat{\eta} &= \frac{1}{\sqrt{\pi}}\int\limits_0^{\infty} \tr^{\ev}\left(\frac{d \A_t}{dt} \exp\left(-\A_t^2\right)\right) dt \\
&= \frac{1}{\sqrt{\pi}} \int\limits_0^{\infty} \tr^{\ev} \left( \left(D-\frac{iT}{4t}\right)\right. \\
& \quad\quad\quad \quad \cdot \left.\exp\left(-tD^2 -\sqrt{t} df- d\beta-ifT +T\n^V_e+ \frac{Dc(T)}{2}+ \frac{T^2}{16t}\right)\right) \frac{dt}{2\sqrt{t}} .
\end{align*}
We see that $df$ is the only odd differential form and because of $df\wedge df=0$ it does not contribute to $\tr^{\ev}$. Since the eigenspaces of $D$ are preserved by all occuring operators, we can write the trace as
\begin{align*}
\hat{\eta}&= \frac{1}{\sqrt{\pi}}\exp\left(-d\beta-ifT\right) \int\limits_0^{\infty} \sum\limits_{k\in \Z} \left(\left(k-f- \frac{iT}{4t}\right)\right.\\
&\quad\quad\quad\quad\quad\quad\quad\quad\quad \left.\exp\left(-t (k+f)^2 +\frac{ (k+f)iT}{2} + \frac{T^2}{16t} \right) \right) \frac{dt}{2\sqrt{t}} \\
&= \frac{1}{\sqrt{\pi}} \exp\left(-d\beta-ifT\right)\int\limits_0^{\infty} \sum\limits_{k\in \Z} \left( \left( k+f-\frac{iT}{4t}\right)\right.\\
& \quad\quad\quad\quad\quad\quad\quad\quad\quad\quad\quad\quad\left.\exp\left(\left( i\sqrt{t}(k+f)+\frac{T}{4\sqrt{t}}\right)^2\right)\right) \frac{dt}{2\sqrt{t}}.
\end{align*}
That is why we have to calculate
\begin{equation*}
\sum\limits_{k\in\Z} \left(k+f-\frac{iT}{4t}\right)\exp\left(\left(i\sqrt{t}(k+f)+\frac{T}{4\sqrt{t}}\right)^2\right) \stackrel{\mathrm{def}}= \sum\limits_{k\in\Z} g(k+f).
\end{equation*}
We denote by $\hat{g}$ the Fourier transform of $g$ and use the generalized Poisson summation formula 
\begin{align*}
\sum\limits_{k\in\Z} g(k+f)&=\sum\limits_{k\in\Z} \hat{g}(k) \cdot \exp\left(2\pi i k f\right)\\
&= -\sum\limits_{k\in\Z} ik \left(\frac{\pi}{t}\right)^{3/2} \exp\left(-\frac{\pi^2k^2}{t}+2\pi i k f +\frac{\pi kT}{2t}\right).
\end{align*}
We insert that into the formula of $\hat{\eta}$ and get
\begin{align*}
\hat{\eta}&= \pi \exp\left(-d\beta-ifT\right)\int\limits_0^{\infty} \sum\limits_{k\in\Z} \frac{k}{i} \frac{1}{t^{3/2}} \exp\left(-\frac{\pi^2 k^2}{t}\right) \exp\left(2\pi i k f + \frac{\pi k i T}{2 it}\right) \frac{dt}{2\sqrt{t}}  \\
&= - \pi \exp\left(-d\beta-ifT\right) \sum\limits_{k=1}^{\infty} \int\limits_0^{\infty} k \exp\left(-\frac{\pi^2 k^2}{t}\right) \sin\left(-2\pi kf + \frac{\pi k iT}{2t}\right) \frac{dt}{t^2}\\
&= -\pi \exp\left(-d\beta-ifT\right)\sum\limits_{k=1}^{\infty} k \int\limits_0^{\infty} \exp\left(-\pi^2 k^2 x\right)\sin\left(-2\pi fk+\frac{\pi k iT}{2} x\right) dx\\
&= -\pi \exp\left(-d\beta-ifT\right)\sum\limits_{k=1}^{\infty}\left(\frac{4k}{4\pi^2 k^2 - T^2} \sin\left(-2\pi f k\right)\right.\\
& \quad\quad\quad\quad\quad\quad\quad\quad\quad\quad\quad\quad\quad\quad\quad\quad \left. + i \frac{2T}{4\pi^3 k^2-\pi T^2 } \cos\left(-2\pi f k\right)\right) \\
&= \exp\left(-d\beta-ifT\right)\left(\sum\limits_{k=1}^{\infty} \sum\limits_{n=0}^{\dim B} \frac{T^{2n}}{2^{2n} \pi^{2n+1} k^{2n+1}} \sin(2\pi f k)\right. \\
&\quad\quad\quad\quad\quad\quad\quad\quad\quad\quad\quad\quad\left. - i \sum\limits_{k=1}^{\infty} \sum\limits_{n=0}^{\dim B} \frac{T^{2n+1}}{2^{2n+1} \pi^{2n+2} k^{2n+2}} \cos(2\pi fk)\right) .
\end{align*}

We define
\begin{equation}
g_n (x) = \begin{cases} \sum\limits_{k=1}^{\infty} \left(2^n \pi^{n+1} k^{n+1}\right)^{-1} \sin \left(2\pi kx\right), \text{ for } n \text{ even } \\
-i \sum\limits_{k=1}^{\infty} \left(2^n \pi^{n+1} k^{n+1}\right)^{-1} \cos \left(2\pi kx\right), \text{ for } n\text{ odd }
\end{cases}
\end{equation}
such that
\begin{equation}
\hat{\eta}= \exp\left(-\beta\right) \sum\limits_n g_n(f) T^n .
\end{equation}
 We see that the functions $g_n$ just depend on $a= f \mod \Z \in [0,1) $. \newline
First of all we look at the case $f(b) \in \Z$ and see immediately that $g_n=0$ for $n\in 2\N$. If $n=2k+1\in 2\N+ 1$ we compute 
\[ g_n(f)= -\frac{i}{2^{n}\pi^{n+1}} \zeta(n+1)=-\frac{i}{2^{2k+1}\pi^{2k+2}}\zeta(2k+2)\]
and therefore
\begin{align*}
\left.\hat{\eta}\right|_{B_0} &= -\exp\left(-d\beta-ifT\right) \sum\limits_{k=0}^{\infty} \frac{i}{2^{2k+1}\pi^{2k+2}} \zeta(2k+2) T^{2k+1} \\
&= -\exp\left(-d\beta-ifT\right)\sum\limits_{k=0}^{\infty} \frac{i^{2k+1}}{(2k+2)!} B_{2k+2} T^{2k+1},
\end{align*}
where $B_i$ are the Bernoulli numbers, i.e. $B_i=\left.\frac{d^i h(x)}{dx^i}\right|_{x=0}$ where $h(x)=\frac{x}{e^x -1}$. We have $B_{2k+1}=0 $ if $k\geq 1$ and get
\begin{align*}
\left.\hat{\eta}\right|_{B_0} &=-\exp\left(-d\beta-ifT\right)\left(iT\right)^{-1}\sum\limits_{k=0}^{\infty}\left. \frac{d^{2k+2}h(x)}{dx^{2k+2}}\right|_{x=0} \frac{1}{(2k+2)!} (iT)^{2k+2} \\
&=\exp\left(-d\beta-ifT\right)(-iT)^{-1} \left(\frac{iT}{e^{iT}-1} -1+\frac{iT}{2}\right).
\end{align*}
For points where $f\not\in\Z$ up to a constant the functions $g_n\colon(0,1)\rightarrow\R$ are the Fourier series of the Bernoulli polynomials
\[ g_n(x)= \frac{(-1)^{n+1}}{i^n (n+1)!} B_{n+1}(x)=-\frac{i^n}{(n+1)!} B_{n+1}(x) .\]
  For Bernoulli polynomials we know that
\[ B_n(x)=\sum\limits_{k=0}^n \binom{n}{k} B_k x^{n-k}, \]
where the $B_k$ are again the Bernoulli numbers. So we get
\begin{align*}
& \left.\hat{\eta}\right|_{B\backslash B_0} =-\exp\left(-d\beta-ifT\right)\sum\limits_{n=0}^{\infty} \frac{1}{(n+1)!} B_{n+1}(a) (iT)^n \\
&=-\exp\left(-d\beta-ifT\right)(iT)^{-1} \sum\limits_{n=0}^{\infty} \sum\limits_{k=0}^{n+1} \frac{1}{k!} \left.\frac{d^k h(x)}{dx^k}\right|_{x=0} (iT)^k \frac{1}{(n+1-k)!} (iaT)^{n+1-k} \\
&=-\exp\left(-d\beta-ifT\right)(iT)^{-1} \left(\left(\sum\limits_{n=0}^{\infty} \frac{1}{n!} \left.\frac{d^n h(x)}{dx^n}\right|_{x=0} (iT)^n\right)\left(\sum\limits_{n=0}^{\infty} \frac{1}{n!} (iaT)^n\right)-1\right) \\
&=\exp\left(-d\beta-ifT\right)(-iT)^{-1}\left(\frac{iT}{e^{iT}-1} \exp\left(iaT\right) -1\right).
\end{align*}
It follows that
\[ \hat{\eta}=\exp\left(-d\beta-ifT\right)(-iT)^{-1} \begin{cases} \left(\frac{-iT}{\exp\left(-iT\right)-1} -1-\frac{iT}{2}\right), & \text{ for }  b\in B_0 \\
\left( \frac{iT}{\exp\left(iT\right)-1}\exp\left(iaT\right)-1\right), & \text{ for } b\in B\backslash B_0 
\end{cases} \]
and
\begin{align*} &\tilde{\eta}=\sum\limits_k \frac{1}{(2\pi i)^k} \hat{\eta}_{[2k]} \\
&= \exp\left(-\frac{d\beta+ifT}{2\pi i}\right)\left(-\frac{T}{2\pi}\right)^{-1} \begin{cases} \left( \frac{-T/2\pi}{\exp(-T/2\pi)-1} -1 -\frac{T}{4\pi}\right), &  b\in B_0 \\
\left( \frac{T/2\pi}{\exp(T/2\pi)-1}\exp\left(\frac{aT}{2\pi}\right)-1\right), & b\in B\backslash B_0.
\end{cases}
\end{align*}

\end{proofthm}

\begin{theorem}
We define $d\tilde{\eta}\colon \Omega^{\bullet}(B) \rightarrow \R $ by
\[ \int\limits_B \left(d\tilde{\eta}\right) \wedge \omega := -\int\limits_B \tilde{\eta} \wedge d\omega. \]
The following formula for the differential holds
\begin{equation}
d \tilde{\eta} = \Int \chern \left( V, \n^V \right) + \delta_{B_0} \chern\left(\ker D \rightarrow B_0, \n^{\ker}\right),
\end{equation}
where $\n^{\ker}=P_0 \n^{\pi_* V} P_0$ and $P_0$ is the projection onto the kernel of $D$.
\end{theorem}

\begin{prf}
We have two different possibilities to calculate the differential of $\tilde{\eta}$. On the one hand we have the transgression formula (\ref{tragre})
\begin{equation}
d \int\limits_s^T \tr^{\ev} \left(\frac{d\A_t}{dt}e^{-\A_t^2}\right) = \tr^{\odd} \left(e^{-A_s^2}\right)- \tr^{\odd}\left(e^{-\A_T^2}\right).
\end{equation}
By \cite[Theorem 2.10]{BismutFreed}  we know the limit for $s\rightarrow 0$ is
\[
\lim_{s\to 0} \frac{1}{\sqrt{\pi}}\tr^{\odd} \left(e^{-\A_t^2}\right)= \left(2\pi i\right)^{-1}\Int \det\left(\frac{R^{M/B}/2}{\sinh \left(R^{M/B}/2\right)}\right)^{1/2} \tr\left(\exp\left(-\left(\n^V\right)^2\right)\right) \]
and since $\AD=\hat{A}\left(TS^1\right)=1$ we get the first term. For the second we need to proof that
\begin{equation}
\lim_{T\to\infty} \tr^{\odd}\left(e^{-\A_T^2}\right)= -\sqrt{\pi}\delta_{B_0} \tr\left(\exp\left(-\left(\n^{\ker}\right)^2\right)\right).
\end{equation}
For that we know that for all eigenvalues $k+f$, $k\neq 0$ and all $\mathcal{C}^{\ell}$-norms
\[ \left\Vert \exp\left(-t (k+f)^2-\sqrt{t}df-d\beta - ifT+i \frac{(k+f)T}{2} + \frac{T^2}{16t}\right) \right\Vert_{\mathcal{C}^{\ell}(B)}\leq C e^{-ct} .\]
For $k=0$ we see that we cannot take the limit as a differential form, we have to integrate over the normal direction of a tubular neighbourhood $N=N_{\varepsilon}\cong B_0\times \left(-\varepsilon, \varepsilon \right)$ of $B_0$ where $f(x,y)=y$. Let $\omega\in\Omega^{\bullet}(B)$ where $\overline{\supp \omega} \subset N$
\begin{align*}
& \int\limits_{-\varepsilon}^{\varepsilon} \exp\left(-ty^2-\sqrt{t}dy -d\beta - iyT+ \frac{iyT}{2}+ \frac{T^2}{16t}\right) \omega \\
& = \int\limits_{-\varepsilon\sqrt{t}}^{\varepsilon\sqrt{t}} \exp\left(-y^2-dy -f_t^* d\beta -\frac{iyf_t^*T}{2\sqrt{t}}+ \frac{f_t^* T^2}{16t} \right) f_t^*\omega
\end{align*}
where $f_t\colon (-\varepsilon\sqrt{t}, \varepsilon\sqrt{t})\rightarrow (-\varepsilon,\varepsilon), x\mapsto\frac{x}{\sqrt{t}}$. Now we can see that we have a Gaussian bell curve and therefore
\begin{align*}
&\lim_{t\to\infty} \int\limits_{-\varepsilon}^{\varepsilon} \exp\left(-ty^2-\sqrt{t}dy-d\beta - \frac{iyT}{2}+ \frac{T^2}{16t}\right) \omega \\
&= - \sqrt{\pi} i^* \exp\left(-d\beta\right) i^*\omega,
\end{align*}
where $i\colon B_0\rightarrow B$ denotes the inclusion. \newline
On the other hand we can directly calculate the formula for $d\tilde{\eta}$ by the formula for $\tilde{\eta}$ of Theorem \ref{etas1} and
\begin{align*}
\int\limits_B \left( d\tilde{\eta}\right) \omega &=- \int\limits_B \tilde{\eta} d\omega  \\
&= -\lim_{\varepsilon \to 0} \int\limits_{B\backslash N} \tilde{ \eta} d\omega  \\
&= \lim_{\varepsilon \to 0} \int\limits_{B\backslash N}  \left( d\tilde{\eta}\right) \omega -\lim_{\varepsilon\to 0} \int\limits_{B\backslash N} d\left( \tilde{\eta} \omega\right)\\
&= \lim_{\varepsilon \to 0} \int\limits_{B\backslash N}  \left( d\tilde{\eta}\right) \omega -\lim_{\varepsilon\to 0} \int\limits_{B_0-\varepsilon} i^*\left(\tilde{\eta}\omega\right) + \lim_{\varepsilon\to 0} \int\limits_{B_0+\varepsilon} i^*\left(\tilde{\eta}\omega\right),
\end{align*}
which will lead to the same formula as the reader may easily check.

\end{prf}

%% file: GeneralCase2.tex
\section{Transversal zero-crossing of an eigenvalue}\label{generalcase}

We will now turn to a more general setting. Let $M\rightarrow B$ be a Riemannian fibre bundle and $V\rightarrow M$ a vertical Dirac bundle as in section \ref{preliminaries}. The transgression formula in \cite[Theorem 4.95]{BC} holds for invertible vertical Dirac operators, it was generalized by \cite[Theorem 10.32]{BGV} for vertical Dirac operators with constant kernel dimension (see also \cite[Theorem 0.1]{Dai} for odd dimensional fibres).  We want to take the next step and give a generalization for a transversal zero-crossing of one eigenvalue of multiplicity one. For the proof we adopt many ideas of the proof of \cite[Theorem 3.2]{BismutComplex}. However, we have to be very careful which norms we use, since our operators are endomorphisms of an infinite rank vector bundle. We also use different contours as in \cite{BismutComplex} which comes from the fact that we want to use holomorphic funtional calculus of the form
\begin{equation}
\exp\left(-\A_t^2\right)=\frac{1}{2\pi i} \int\limits_{\Gamma} \frac{e^{-z}}{z-\A_t^2} dz
\end{equation}
rather than
\begin{equation}
\exp\left(-\A_t^2\right)=\frac{1}{2\pi i} \int\limits_{\tilde{\Gamma}} \frac{e^{-z^2}}{z-\A_t} dz.
\end{equation}

\begin{ass}\label{assumption}
We assume that we can find a covering $\{ U_i \}_{1\leq i \leq k} $ of $B$ such that on each $U_i$ either $D_b$ is invertible or we have a smooth function $f_i:U_i \rightarrow (-K,K)$ which has $0$ as a regular value, such that $\spec D_b \cap [-K -\delta, K+\delta ]=\{ f_i(b)\}$ and $f_i(b)$ is an eigenvalue of multiplicity $1$.
\end{ass}

\begin{rem}
We get a codimension $1$ submanifold 
\[ B_0= \bigcup f_i^{-1}\left(\{0\}\right) \subset B\]
where we have a complex line bundle $\ker D \rightarrow B_0$ and $D_b$ is invertible for all $b\in B\setminus B_0$. We denote by $i\colon B_0\rightarrow B$ the inclusion. As in section \ref{example} we get an orientation on $B_0$ by
\[ \left(v_2,..., v_{m-1}\right)\in o_x(B_0) \Leftrightarrow \left( \grad_x f, v_2,..., v_{m-1}\right)\in o_x(B).\]
Let $\nu B_0 \rightarrow B_0$ be the normal bundle which is trivial $\nu B_0\cong B_0\times \R$ in our situation. Then we find a constant $0<\eps\leq K$ small enough such that
\[ \exp\colon B_0\times (-\eps, \eps)\rightarrow B \]
is a diffeomorphism onto its image $N_{\eps}$. We will not fix $\eps$ since we may take it as small as needed in the proofs. Without loss of generality we may assume that under this identification 
\[ f\left(x,y\right)=y. \]
To achieve that we maybe need to change the metric on $B$ but we know by \cite[Proposition 10.2]{BGV} that $\n^{M/B}$ is independent of $g_B$ and there is also a formula for $T(U,V)=-P[U,V]$ which is independent of the metric on $B$. 
\end{rem}

\begin{propdef}
Let $P_b$, $b\in N_{\eps}$ be the orthogonal projection onto the spectral subspace $(-\eps-\delta, \eps+\delta)$ of $D_b$. Then
\[ L=\im P \rightarrow N_{\eps} \]
is a smooth line bundle on the tubular neighbourhood $N_{\eps}$ of $B_0$. We denote the projection onto the orthogonal complement $W$ by $Q=1-P$ and the projection of the connection $\n^{\pi_* V}$ onto the subbundles $L$ and $W$ by 
\[ \n^{L\oplus W}=P\n^{\pi_* V} P \oplus Q \n^{\pi_* V} Q .\]
The projections of $D$ are denoted by $D^{-}=DP= y P$ and $D^+ = DQ$. 
\end{propdef}
\begin{proof}
This follows from \cite[Proposition 9.10]{BGV} since $\pm\eps\pm\delta$ is not an eigenvalue of $D_b$ for $b\in N_{\eps}$.
\end{proof}

\begin{lemma}\label{trivpiv}
Locally on $N_{\eps}\cong B_0 \times (-\eps, \eps)$ we trivialize $\pi_* V$ along normal geodesics by parallel transport with respect to the connection $\n^{\pi_* V}$. (Note, that it is in general not possible to trivialize with respect to the connection $\n^{L\oplus W}$.)
\end{lemma}
\begin{proof}
For $b\in B_0$ the lifts of the geodesic $\exp_b\colon (-\eps, \eps)\rightarrow N_{\eps}$ gives a family of geodesics $\widetilde{\exp_b}\colon M_b \times (-\eps, \eps)\rightarrow \pi^{-1}(N_{\eps})$, see \cite[Corollary 1.11.11]{Klingenberg}. By taking $\eps$ small enough we may assume that $\widetilde{\exp_b}(\cdot, t)\colon M_b \rightarrow M_{\exp_b (t)}$ is an isomorphism for all $t\in (-\eps, \eps)$. Therefore if $\sigma\in (\pi_* V)_b = \Gamma\left(M_b, \left. V\right|_{M_b}\right)$ we can use parallel transport for each $\sigma_x \in V_{b,x}$ with respect to the connection $\n^V + \frac{1}{2}k$ to get a vector in $V_{\widetilde{\exp_b}(x,t)}$. This depends smoothly on $x\in M_b$ so we get a smooth section in $\left( \pi_* V \right)_{\exp_b(t)}$.
\end{proof}

\begin{defn}
We denote by
\[ E_t :=\A_t^2-tD^2= \sqrt{t} [D,\n^{\pi_* V}] +\left(\n^{\pi_* V}\right)^2 -\frac{[D , c(T)]}{4}- \frac{[\n^{\pi_* V},c(T)]}{4 \sqrt{t}} +\frac{c(T)^2}{16t} . \]
By our assumption
\[ \exists \tilde{K}>0: \sup_{(x,y)\in N} f^2(x,y)+\tilde{K}=\varepsilon^2+\tilde{K}\leq \inf_{(x,y)\in N} \lambda^2_k(x,y) \quad \forall k \neq 0 ,\]
where $\lambda_k, k\neq 0$ denote all the other eigenvalues of $D$ which do not cross zero.
Let $K:= \varepsilon^2 +\frac{\tilde{K}}{2}$ and define the contours $\Omega_t, \Gamma_t\in  \CC$ 
\begin{center}
 \begin{tikzpicture}
   \node[anchor= south east] at (3,0) {$Kt$};
   \node[left] at (0,1) {$i$};
   \node[left] at (0,-1) {$-i$};
\draw [->](-1.5,0)  -- (9,0) node(xline)[right]{Re};
\draw[->](0,-1.5)--(0,1.5) node(yline)[above]{Im};  
\draw[->] (3,1)--(3,-0.5);
\draw (3, -0.5)--(3,-1);
\draw (3,1)--(5,1);
\draw[<-] (5,1)--(7,1);
\draw[->] (3,-1)--(5,-1);
\draw(5,-1)--(7,-1);
\draw[dotted](7,1)--(8,1) node[above]{$\Gamma_t$};
\draw[dotted](7,-1)--(8,-1); 
 \end{tikzpicture}

 \begin{tikzpicture}
 \draw [->](-1.5,0)  -- (9,0) node(xline)[right]{Re};
\draw[->](0,-1.5)--(0,1.5) node(yline)[above]{Im};  
 \node[anchor= south east] at (0,1) {$i$};
  \node[anchor= south east] at (0,-1) {$-i$};
  \node[anchor=south east] at (3,0){$Kt$};
  \node[anchor=south east] at (-1,0){$-1$};
 \draw[->](-1,1)--(-1,-0.5);
 \draw(-1,-0.5)--(-1,-1);
 \draw[<-](1.5,1)--(3,1) node[above]{$ \Omega_t$};
 \draw(1.5,1)--(-1,1);
 \draw [->](-1,-1)--(1.5,-1);
 \draw(1.5,-1)--(3,-1);
 \draw[->] (3,-1)--(3,-0.5);
 \draw(3,-0.5)--(3,1);
 
 \end{tikzpicture}

\end{center}
such that the small eigenvalue $tf^2(x,y)=ty^2$ of $tD^2$ lies inside the contour $\Omega_t$ and the large eigenvalues $t\lambda_k^2$ lie inside $\Gamma_t$.\newline
Since
\begin{equation}
\left(z-\A_t^2\right)^{-1}=\sum\limits_{n=0}^{\dim B} \left(z-tD^2\right)^{-1} \left(E_t\left(z-tD^2\right)^{-1}\right)^n
\end{equation}
the spectrum of $\A_t^2$ equals the spectrum of the rescaled Dirac operator . On $B_0\times \left(-\varepsilon, \varepsilon\right)$ we have $\sigma\left(\A_t^2\right)=\sigma\left(tD^2\right)=\{ t\lambda_k^2\}_{k\in\Z}$. By holomorphic functional calculus \cite[Chapter XV, Proposition 1.1]{Gohberg} we know that on $N_{\eps}$
\begin{align*}
\exp\left(-\A_t^2\right)& =\frac{1}{2\pi i} \int\limits_{\Omega_t \cup \Gamma_t} \exp\left(-z\right) \left(z-\A_t^2\right)^{-1} dz \\
&= \frac{1}{2\pi i}\int\limits_{\Omega_t} \exp\left(-z\right)\left(z-\A_t^2\right)^{-1}dz + \frac{1}{2\pi i} \int\limits_{\Gamma_t}\exp\left(-z\right)\left(z-\A_t^2\right)^{-1} dz \\
&= \PP\left(\exp\left(-\A_t^2\right)\right) + (1-\PP)\left(\exp\left(-\A_t^2\right)\right).
\end{align*}
Note that the projection 
\begin{equation*}
\PP = \frac{1}{2\pi i}\int\limits_{\Gamma_t}\left(z-\A_t^2\right)^{-1}dz\colon\Lambda^{\bullet}T^*B\otimes \pi_* V\rightarrow\Lambda^{\bullet}T^*B\otimes\pi_* V
\end{equation*}
coincides in degree $0$ with the spectral projection $P\colon\pi_* V\rightarrow L\subset \pi_* V$.
\end{defn}

\begin{defn}
We take the pullback of the bundle $\ker D\rightarrow B_0$ of $\pi_1\colon B_0\times\R\rightarrow B_0$ with the connection $\pi_1^* \n^{\ker}$ which, by abuse of notation, will also be denoted by $\n^{\ker}$. We denote the second coordinate of $B_0\times\R$ by $y$ and consider the superconnection 
\[ y+\n^{\ker}\colon \Omega^{\bullet}\left(B_0 \times\R , \pi_1^*\ker D \right) \rightarrow \Omega^{\bullet}\left(B_0\times\R , \pi_1^* \ker D\right) ,\]
where we assume that $y$ and $1$-forms anticommute. Note that this differs slightly from the superconnection $B$ introduced in \cite[Section III.a]{BismutComplex}. \newline
If $\left\vert y\right\vert\leq \eps\sqrt{t}$ we can proceed as in the previous definition and write 
\[ \exp\left(-\left(y+\n^{\ker}\right)^2\right)=\frac{1}{2\pi i} \int\limits_{\Omega_t} \exp(-z) \left(z- \left(y+\n^{\ker}\right)^2\right)^{-1} dz .\]
\end{defn}

\begin{notation}
We will need different kinds of norms in the following statements and proofs which we will introduce here. See also \cite[Appendix of IX.4, Example 2]{ReedSimon}. \newline
We denote by $H^k=W^{(k,2)}\left(M_b, V_b \right) $ the $k$th Sobolev space of sections with Sobolev norm $\left\vert\cdot\right\vert_k$, $H^0=L^2\left(M_b, V_b\right)$. For a linear operator  $A:H^k \rightarrow H^{k'}$ we define the operator norm
\begin{equation}
\left\Vert A \right\Vert_{k, k'}= \sup_{\left\vert x \right\vert_{k}} \left\vert A(x) \right\vert_{k'} .
\end{equation}
We say a bounded linear operator $A\in\mathcal{L}\left(H^0\right)$ is \textit{trace-class} if
\begin{equation}
\left\Vert A \right\Vert_1 = \tr \left\vert A \right\vert < \infty.
\end{equation}
For $1 \leq p < \infty$ the \textit{$p$-Schatten norm} is defined by
\begin{equation}
\left\Vert A \right\Vert_p =\left(\tr\left( \left\vert A\right\vert^p\right)\right)^{1/p}.
\end{equation}
For a smooth differential form $\omega \in \Omega^{\bullet}(B)$ we denote by $\left\Vert\omega\right\Vert_{\mathcal{C}^{\ell}}$ the $\mathcal{C}^{\ell}$-norm. For $\omega\in\Omega^{\bullet}(B_0\times (-\eps, \eps))$ we see $\left\Vert\omega\right\Vert_{\mathcal{C}^{\ell}(B_0)}$ as a function on $(-\eps, \eps)$.

\end{notation}

\begin{rem}
The trivialization of Lemma \ref{trivpiv} provides us with an isometry 
\[ L^2\left( M_x, V_x\right) \cong L^2 \left(M_{(x,y)}, V_{(x,y)}\right)\]
for all $\left(x,y\right)\in B_0 \times (-\eps,\eps)$. If we work with Sobolev-sections for $k>0$ we still get an isomorphism but not an isometry. However we know that the topology of the Banach spaces is the same and therefore the Sobolev norms are equivalent. In particular since $B_0$ is compact and if $\eps$ is small enough we find constants $C, c>0$ such that for all $\left(x,y\right)\in B_0 \times (-\eps,\eps)$ and all sections $\sigma \in W^{k,2}\left(M_x, V_x\right)\cong W^{k,2}\left(M_{(x,y)}, V_{(x,y)}\right)$ the following estimate holds true
\[ C \left\vert \sigma \right\vert_{k, (x,y)} \leq \left\vert\sigma\right\vert_{k, x} \leq c \left\vert \sigma \right\vert_{k, (x,y)} .\]
So in the following estimates we will make no difference for which $y\in (-\eps,\eps)$ we use the Sobolev norms because by changing the constants the estimates hold for all points $y$ and we get the same speed of convergence. 
\end{rem}

\begin{lemma}\label{bla}
Let $z \in \Gamma_t$ or $z\in\Omega_t$, $p\geq \dim M_b +1$ and $t$ big enough, then we have the following estimates: 
\begin{equation} \label{absch1}
\left\Vert \left(z-tD_b^2\right)^{-1} \right\Vert_{0,0} \leq C_1,
\end{equation}
\begin{equation}\label{absch2}
  \left\Vert \left( z-tD_b^2\right)^{-1}\right\Vert_p \leq C_2 \left(1+\frac{\left\vert z\right\vert}{t}\right),
\end{equation} 
\begin{equation}\label{absch4}
\left\Vert \left(z-tD_b^2\right)^{-1}\right\Vert_{0,2}\leq C_3 \left(1+\frac{\left\vert z \right\vert}{t}\right),
\end{equation}
for every $b\in N_{\eps}$.
\end{lemma}

\begin{prf}
 (\ref{absch1}) follows from the choice of the contours $\Gamma_t$ and $\Omega_t$. \newline
(\ref{absch2}) and (\ref{absch4}) follow as in \cite[Proposition 7.2]{BismutGoette} by writing
\[ \left(z-tD^2\right)^{-1}=  t^{-1}\left(i-D^2\right)^{-1} - \left(i-D^2\right)^{-1} \left(\frac{z}{t} -i\right)\left(z-tD^2\right)^{-1} .\]
We then use the well-known facts that there exist constants such that
\[ \left\Vert \left(i-D^2\right)^{-1}\right\Vert_p \leq C \]
for $k\geq \dim M_b +1$, this follows for example by  \cite[Remark 5.32, Proposition 8.9]{Roe}, and
\[ \left\Vert \left(i-D^2\right)^{-1}\right\Vert_{0,2}\leq C \]
see \cite[Equation (7.7)]{BismutGoette}. Together with estimate (\ref{absch1}) these prove the claimed inequalities (\ref{absch2}) and (\ref{absch4}).
\end{prf}

\begin{prop}\label{1-P}
On the tubular neighbourhood $N_{\eps}\cong B_0\times \left(-\varepsilon, \varepsilon\right)$ of $B_0$ in $B$ there exist for all $\ell\geq 0$ constants $c,C>0$ and a polynomial $f \in \R[t, t^{-1}]$ such that
\begin{equation}
\left\Vert \tr^{\odd} \left( (1-\PP)\left( \exp\left( -\A^2_t\right) \right)\right)\right\Vert_{\mathcal{C}^{\ell}(N)} \leq  c f(t) \exp\left(-Ct\right).
\end{equation}

\end{prop}

\begin{prf}
We will first prove the statement for $\ell=0$.\newline
By the definition of the operator $E_t$ and since $B$ is compact we know that
\[ \left\Vert E_t\right\Vert_{2,0}\leq C\sqrt{t} .\]
Combining this with the estimates (\ref{absch2}) and (\ref{absch4}) we get
\begin{align*}
\left\Vert \left( z-\A_t^2\right)^{-p}\right\Vert_1 &\leq \left\Vert\left(z-\A_t^2\right)^{-1}\right\Vert^p_p \\
& \leq \left( \sum\limits_{n=0}^{\dim B} \left\Vert \left(z-tD^2\right)^{-1}\right\Vert_{0,2} \left\Vert E_t\right\Vert^n_{2,0} \left\Vert \left(z-tD^2\right)^{-1}\right\Vert^n_p  \right)^p \\
&\leq \left( \sum\limits_{n=0}^m C \left(1+\frac{|z|}{t}\right) t^{n/2} \right)^p \\
& \leq C \left(1+\frac{|z|}{t}\right)^p t^{mp/2},
\end{align*}
where $m=\dim B$ and constants $C$ varying from line to line. It follows that
\begin{align*}
&\left\Vert \tr^{\odd}\left((1-\PP)\left(\exp\left(-\A_t^2 \right)\right)\right)\right\Vert_{\mathcal{C}^0}\\
 &\leq \left\Vert(1-\PP)\left(\exp\left( -\A_t^2\right)\right)\right\Vert_{1} \\
& = \left\Vert \frac{1}{2\pi i} \int\limits_{\Gamma_t} \frac{\exp (-z)}{z-\A_t^2} dz \right\Vert_1 \\
&= \frac{1}{2\pi p!} \left\Vert \int\limits_{\Gamma_t} \frac{\exp (-z)}{\left( z-\A_t^2\right)^p} dz \right\Vert_1 \\
& \leq \frac{1}{2\pi k!} \int\limits_{\Gamma_t} \left\vert\exp(-z)\right\vert C \left(1+\frac{\left\vert z\right\vert}{t}\right)^p t^{mp/2} dz\\
& \leq C f(t) \exp (-Kt),
\end{align*}
where $f\in\R[t,t^{-1}]$.\newline
For $\ell \geq 1$ we first see by the same argument as above and Lemma \ref{bla} that
\begin{align*}
\left\Vert \left(z-\A_t^2\right)^{-1}\right\Vert_{0,2} & \leq \sum_{n=0}^m \left\Vert\left(z-tD^2\right)^{-1}\right\Vert_{0,2} \left\Vert E_t\right\Vert^n_{2,0}\left\Vert\left( z-tD^2\right)^{-1}\right\Vert_{0,2}^n \\
&\leq C t^{m/2}\left(1+\frac{\left\vert z\right\vert}{t}\right)^m .
\end{align*}
Let $\n$ be any connection on $\pi_* V$. We know by \cite[Lemma 9.15]{BGV} that for a family of smoothing operators $K$, $d\tr\left( K\right)=\tr\left( \n(K)\right)$ independent of the connection and by \cite[Theorem 9.51]{BGV} $\exp\left(-\A_t^2\right) $ is a family of smoothing operators. For local coordinates $y_1,..., y_m$ on $B$ it is clear that
\[  \left\Vert \n_{\frac{\del}{\del y_{i_1}}} ...\n_{\frac{\del}{\del y_{i_k}}}  \left(\A_t^2\right) \right\Vert_{2,0}\leq C t .\]
Now by using 
\begin{equation*}
\n_{\frac{\del}{\del y_i}} \left(z-\A_t^2\right)^{-1}= \left(z-\A_t^2\right)^{-1}\n_{\frac{\del}{\del y_i}}\left(\A_t^2\right)\left(z-\A_t^2\right)^{-1} 
\end{equation*}
one can prove that
\begin{align*}
\left\Vert \n_{\frac{\del}{\del y_i}}\left(z-\A_t^2\right)^{-1}\right\Vert_p & \leq \left\Vert \left(z-\A_t^2\right)^{-1}\right\Vert_p \left\Vert \n_{\frac{\del}{\del y_i}}\left(\A_t^2\right)\right\Vert_{2,0} \left\Vert\left(z-\A_t^2\right)^{-1}\right\Vert_{0,2} \\
& \leq C p_1\left(\left\vert z\right\vert\right) p_2\left( t\right)
\end{align*}
for $p_1$ a polynomial in $\left\vert z\right\vert$ and $p_2$ a polynomial in $t$ and $t^{-1}$. It follows inductively that
\[ \left\Vert \n_{\frac{\del}{\del y_{i_1}}}...\n_{\frac{\del}{\del y_{i_k}}} \left(z-\A_t^2\right)^{-1}\right\Vert_p \leq C p_1\left(\left\vert z\right\vert\right) p_2\left(t\right) .\]
This finishes the proof as in the case $\ell =0$ above.

\end{prf}

\begin{defn}
We define the functions $g$, $f_t$ and $i$ to be
\begin{align*}
g\colon B_0 \times (-\eps\sqrt{t}, \eps\sqrt{t})& \rightarrow B_0 , (x,y)\mapsto x , \\ 
f_t\colon B_0 \times (-\eps\sqrt{t}, \eps\sqrt{t})& \rightarrow B_0 \times(-\eps,\eps), (x,y)\mapsto \left(x, \frac{y}{\sqrt{t}}\right)
\end{align*}
and
\begin{equation*}
i\colon B_0 \rightarrow B_0 \times (-\eps,\eps), x \mapsto (x,0).
\end{equation*}
For $y\in (-\eps\sqrt{t}, \eps\sqrt{t})$ and $\left\vert y\right \vert \geq 1$ the contour $\Theta_y \subset\CC$ is defined to be
\begin{center}
 \begin{tikzpicture}
  \draw [->](-1.5,0)  -- (9,0) node(xline)[right]{Re};
\draw[->](0,-1.5)--(0,1.5) node(yline)[above]{Im};  
 \node[left] at (0,1) {$i$};
  \node [left]at (0,-1) {$-i$};
  \node[anchor=south east] at ( 1,0) {$\frac{y^2}{2} $};
  \node[anchor=south east] at(4,0){$K'y^2$};
  \node[above]at(2,0) {$y^2$};
  \draw[->] (1,1)--(1,-0.5);
  \draw (1,-0.5)--(1,-1);
  \draw [->](4,1) node[above]{$\Theta_y $}--(2,1);
  \draw(2,1)--(1,1);
  \draw[->](1,-1)--(2,-1);
  \draw(2,-1)--(4,-1);
  \draw[->](4,-1)--(4,-0.5);
  \draw(4,-0.5)--(4,1);
 \end{tikzpicture}
 \end{center}
such that it contains the small eigenvalue of $t D^2_{(x,y/\sqrt{t})}$ for all $x\in B_0$. Then we can write the spectral projection $\PP$ also as
\begin{equation}
\PP\left(\exp\left(-f_t^*\A_t^2\right)_{(x,y)}\right)=\frac{1}{2\pi i} \int\limits_{\Theta_y} \exp (-z) \left(z-f_t^*\A_t^2\right)^{-1} dz
\end{equation}
as well as
\begin{equation}
\exp\left(-\left(y+\n^{\ker}\right)^2\right)=\frac{1}{2\pi i} \int\limits_{\Theta_y}\exp(-z) \left(z- \left( y+\n^{\ker}\right)^2\right)^{-1} dz.
\end{equation}
\end{defn}

\begin{rem}
It is clear by the definition of the contour $\Theta_y$ that the estimates in Lemma \ref{bla} also hold for $z\in\Theta_{y\sqrt{t}}$.
\end{rem}

\begin{lemma}\label{zurueckholen}
If $\omega$ is a differential form on $B$ with support in $B_0\times (-\eps,\eps)$ and $\alpha$ a multiindex of length $\ell$ then
\begin{equation}
\left\vert D^{\alpha} \left((i \circ g)^* \omega- f_t^* \omega\right)_{(x,y)} \right\vert \leq \frac{C}{\sqrt{t}} \left\Vert\omega\right\Vert_{\mathcal{C}^{\ell +1}(B)} \left(1+\left\vert y\right\vert \right). 
\end{equation}
\end{lemma}
\begin{proof}
This follows by a straight-forward calculation and the mean value theorem, see also \cite[Eq. (3.107)]{BismutComplex} for the statement.
\end{proof}

\begin{lemma}\label{lemma}
Let $(x,y)\in B_0\times (-\eps,\eps) $ and $z \in \Omega_t$ or $z\in\Theta_{y\sqrt{t}}$, $\eps$ small enough and $t$ big enough. By abuse of notation we write $\left(D^+_{(x,y)}\right)^{-1}$ instead of $\left(D_{(x,y)}^+\right)^{-1} Q_{(x,y)}$. Then the following inequalities hold
\begin{align*}
\left\Vert \left( z-t\left(D_{(x,y)}^+\right)^2\right)^{-1}\right\Vert_{0,2} &\leq \frac{C}{t}\left(1+\left\vert z \right\vert \right), \\
\left\Vert \left(z-t \left(D^+_{(x,y)}\right)^2 \right)^{-1} + t^{-1} \left(D_{(x,0)}^+\right)^{-2}  \right\Vert_{0,2} & \leq C t^{-1} \left( \left\vert y\right\vert + t^{-1} \left\vert z\right\vert+ t^{-1} \left\vert z \right\vert^2\right).
\end{align*}

\end{lemma}

\begin{prf}
The proof follows the ideas of the proof of \cite[Proposition 3.4]{BismutComplex}. Our constants $C>0$ may vary from line to line but they are all indepenent of $t,y$ and $z$ and since $B_0$ is compact also of $x$. \newline
For the first estimate we write
\begin{equation}\label{1} 
\left(z-t\left(D^+_{(x,y)}\right)^2\right)^{-1}= -t^{-1}\left(1-\frac{z}{t}\left(D^+_{(x,y)}\right)^{-2}\right)^{-1} \left(D_{(x,y)}^+\right)^{-2}.
\end{equation}
As in \cite[Eq. (3.37)]{BismutComplex} we know that for $\left\vert \IM z\right\vert=1$
\begin{align}
\left\Vert \left(1-\frac{z}{t} \left(D^+_{(x,y)}\right)^{-2}\right)^{-1} \right\Vert_{0,0}&\leq \sup_{x\in\R} \left\vert 1-x z\right\vert^{-1} \\
& = \frac{1}{\inf_{x\in\R}\left\vert 1-xz\right\vert} \\
&= \left\vert z\right\vert .
\end{align}
If $\left\vert\IM z\right\vert < 1$ we know that either $\RE z = Kt$, $\RE z = -1$ or $\RE z = Cty^2$. We find a constant $C>0$ such that for $t$ big enough in each of these three cases
\begin{equation}
\left\Vert \frac{\RE z}{t}\left(D^+_{(x,y)}\right)^{-2}\right\Vert_{0,0}\leq C \eps^2,
\end{equation}
in particular for $\eps$ small enough
\begin{equation}
\left\Vert \frac{\RE z}{t} \left(D^+_{(x,y)}\right)^{-2}\right\Vert_{0,0}\leq \frac{1}{2}
\end{equation}
and therefore
\begin{equation}
\left\Vert \left(1-\frac{z}{t}\left(D^+_{(x,y)}\right)^{-2}\right)^{-1}\right\Vert_{0,0}\leq 2.
\end{equation}
So for all $z$ in the contours $\Omega_t$ and $\Theta_{y\sqrt{t}}$ the inequality
\begin{equation}
\left\Vert \left(1- \frac{z}{t}\left(D^+_{(x,y)}\right)^{-2}\right)^{-1}\right\Vert_{0,0}\leq C \left(1+\left\vert z\right\vert\right)
\end{equation}
holds true. Also for $\eps$ small enough we find a constant $C>0$ sucht that for all $(x,y)\in B_0\times (-\eps,\eps)$
\begin{equation}
\left\Vert \left(D^+_{(x,y)}\right)^{-2}\right\Vert_{0,2}\leq C .
\end{equation}
Inserting this into equation (\ref{1}) leads to
\begin{equation}
\left\Vert \left(z-t\left(D^+_{(x,y)}\right)^2\right)^{-1} \right\Vert_{0,2}\leq \frac{C}{t}\left(1+\left\vert z\right\vert\right)
\end{equation}
which completes the first part of the lemma. \newline
For the second inequality of the lemma we write
\begin{align*}
& \left\Vert \left(z-t\left(D^+_{(x,y)}\right)^2\right)^{-1}+t^{-1}\left(D_{(x,y)}^+\right)^{-2}\right\Vert_{0,2} \\
& \leq \left\Vert \left(z-t\left(D^+_{(x,y)}\right)^2\right)^{-1} \frac{z}{t} \left(D^+_{(x,y)}\right)^{-2} \right\Vert_{0,2} + \left\Vert t^{-1} \left(D^+_{(x,0)}\right)^{-2} - t^{-1}\left( D_{(x,y)}^+\right)^{-2} \right\Vert_{0,2}
\end{align*}
By \cite[Satz 2.8]{Ruzicka} we know that
\begin{equation}
t^{-1}\left\Vert \left(D^+_{(x,0)}\right)^{-2} - \left(D^+_{(x,y)}\right)^{-2}\right\Vert_{0,2}\leq \frac{C}{t} \left\vert y\right\vert 
\end{equation}
and by using the first part we have
\begin{equation}
\left\Vert \left(z-t\left(D^+_{(x,y)}\right)^2\right)^{-1} \frac{z}{t} \left(D^+_{(x,y)}\right)^{-2}\right\Vert_{0,2}\leq \frac{C}{t^2}\left(\left\vert z\right\vert + \left\vert z\right\vert^2\right).
\end{equation}
Combing these leads to
\begin{equation}
\left\Vert \left(z-t\left(D^+_{(x,y)}\right)^2\right)^{-1}-t^{-1}\left(D^+_{(x,0)}\right)^{-2}\right\Vert_{0,2}\leq \frac{C}{t}\left(\left\vert y\right\vert + t^{-1}\left\vert z\right\vert + t^{-1}\left\vert z\right\vert^2\right)
\end{equation}
which completes the second part of the lemma.  
\end{prf}

\begin{prop}[{\cite[Proposition 3.5]{BismutComplex}}] \label{zushg}
For $x\in B_0$ and $X\in T_x B$ 
\begin{equation}
\n^{\pi_* V}_X - \n^{L\oplus W}_X = \begin{pmatrix}
0 & P  \n^{\pi_*V}_X (D) Q \left(D^+\right)^{-1}\\
- \left(D^+\right)^{-1} Q \n^{\pi_*V}_X (D) P & 0 
\end{pmatrix}
\end{equation}
with respect to the decomposition $\left.\pi_* V\right|_{B_0} = \ker D \oplus \im D$. Therefore 
\begin{equation}
\left(\n^{\ker}\right)^2_x= P\left(\n^{\pi_* V}\right)^2P - P \n^{\pi_* V}(D) \left(D^{+}\right)^{-2}  \n^{\pi_* V}(D) P .
\end{equation}
\end{prop}

\begin{prop}\label{prop1}
We define for $\left(x,y\right)\in B_0 \times \left(-\varepsilon\sqrt{t},\varepsilon\sqrt{t}\right)$, $z\in\Omega_t$ or $\Theta_y$ the operator $\alpha$ by
 \begin{align*}
& \left.\left( P f_t^* E_t P + P f_t^* E_t Q \left(z-t f_t^* D^2\right)^{-1} Q E_t P \right)\right|_{(x,y)}\\
& = \left. g^* \left( dy + \left(\n^{\ker}\right)^2 \right)\right|_{(x,y)}+ \alpha\left(x,y,z,t\right)
\end{align*}
where we identify $L_{(x,y/\sqrt{t})}$ and $\ker D_{(x,0)}$ by parallel transport along the geodesic $s\mapsto (x, s y /\sqrt{t})$ with respect to $\n^L$. Then there exists a constant $C>0$ such that for $t$ big enough
\begin{equation}
\left\Vert \alpha\left(x,y,z,t\right) \right\Vert_{2,0} \leq C t^{-1/2}\left(1+\left\vert y\right\vert + \left\vert z \right\vert + \left\vert z \right\vert^2\right) .
\end{equation}
\end{prop}

\begin{prf}
 
First we use Proposition \ref{zushg} to see that
\begin{align*}
& \left\Vert Pf_t^*E_tP+Pf_t^*E_tQ\left(z-tf_t^*D^2\right)^{-1} Qf_t^*E_tP-  g^* \left(dy+\left(\n^{\ker}\right)^2\right)\right\Vert_{2,0} \\
&=\left\Vert Pf_t^*E_tP+Pf_t^*E_tQ\left(z-tf_t^*D^2\right)^{-1} Qf_t^*E_tP \right. \\
& \quad \left. - g^*\left(dy + P\left(\n^{\pi_* V}\right)^2P - P\n^{\pi_* V}(D)\left(D^{+}\right)^{-2} \n^{\pi_* V}(D)P\right)\right\Vert_{2,0} \\
& \leq \left\Vert Pf_t^* E_t P - g^* \left(dy+P\left(\n^{\pi_* V}\right)^2P\right)\right\Vert_{2,0} \\
& \quad + \left\Vert   Pf_t^* E_t Q \left(z-tf_t^* D^2\right)^{-1}Qf_t^*E_tP + g^*\left( P\n^{\pi_* V} (D)\left(D^{+}\right)^{-2} \n^{\pi_*}(D)P\right) \right\Vert_{2,0} .
\end{align*}
By definition, Lemma \ref{zurueckholen} and \cite[Satz 2.8]{Ruzicka} 
\begin{equation}
\left\Vert Pf_t^* E_t P-g^*\left(dy+P\left(\n^{\pi_* V}\right)^2P\right)\right\Vert_{2,0} \leq \frac{C}{\sqrt{t}}\left(1+\left\vert y\right\vert\right)  .
\end{equation}
For the second summand we have
\begin{align*}
& \left\Vert Pf_t^* E_t Q \left(z-t D^2_{(x, y/\sqrt{t})}\right)^{-1} Q f_t^* E_t P +  g^* \left(P\n^{\pi_* V} (D) \left( D^+_{(x,0)}\right)^{-2}  \n^{\pi_* V} (D)P\right)\right\Vert_{2,0} \\
& \leq \left\Vert Pf_t^* E_t Q \left(z-t D^2_{(x,y/\sqrt{t})}\right)^{-1}Q \left(f_t^* E_t - \sqrt{t} \n^{\pi_* V}(D) \right)P  \right\Vert_{2,0} \\
& \quad + \left\Vert Pf_t^* E_t Q \left( \left(z-tD^2_{(x,y/\sqrt{t})}\right)^{-1} + t^{-1} \left(D^+_{(x,0)}\right)^{-2} \right) Q \sqrt{t} \n^{\pi_* V}(D)P \right\Vert_{2,0} \\
& \quad + \left\Vert P \left(-f_t^* E_t + \sqrt{t} \n^{\pi_* V}(D)\right) t^{-1} \left(D^+_{(x,0)}\right)^{-2} \sqrt{t} \n^{\pi_* V} (D) P\right\Vert_{2,0} \\
& \leq C_1 t^{-1/2} \left(1+\left\vert z \right\vert+ \left\vert z\right\vert^2\right) + C_2 t^{-1/2} \left( \left\vert y\right\vert + \left\vert z\right\vert + \left\vert z\right\vert^2\right) + C_3 t^{-1/2}
\end{align*}
where we used Lemma \ref{lemma} and the definition of $E_t$. 
\end{prf}

\begin{prop}\label{Normabsch}
Let $\left(x,y\right)\in B_0 \times \left(-\varepsilon \sqrt{t},\varepsilon\sqrt{t}\right)$, $z$ in one of our contours and $t$ big enough. We define
\begin{equation}
\left(z-f_t^* \A_t^2\right)^{-1} - \left(z- \left(y+\n^{\ker}\right)^2\right)^{-1} =: \gamma\left(x,y,z,t\right).
\end{equation}
Then there exist constants $C_1, C_2, C_3, C_4 >0$ and polynomials $p_1,p_2,p_3, p_4, p_5$ such that
\begin{align*}
\left\Vert P\gamma P\right\Vert_{0,2} & \leq C_1 t^{-1/2}\left(1+p_1\left(\left\vert y\right\vert\right) + p_2\left(\left\vert z\right\vert\right)\right) \\
\left\Vert P\gamma Q\right\Vert_{0,2} &\leq C_2 t^{-1/2} \left(1+p_3\left(\left\vert z\right\vert\right)\right)\\
\left\Vert Q \gamma P\right\Vert_{0,2}&\leq C_3 t^{-1/2} \left(1+ p_4\left(\left\vert z\right\vert\right)\right) \\
\left\Vert Q\gamma Q\right\Vert_{0,2} & \leq C_4 t^{-1}\left(1+p_5\left(\left\vert z\right\vert\right)\right).
\end{align*}

\end{prop}

\begin{prf}
Throughout the proof we will denote by $p$ some polynomial in $\left\vert z\right\vert$ or $\left\vert y\right\vert$ which may vary from line to line but is independent of $x,t$ and $y$ or $z$ respectively. The constants $C >0$ may also vary but again are indepenent of $x,y,z$ and $t$. For simplicity but by abuse of notation we define just for this proof $A:= \left(z-tf_t^* D^2\right)^{-1}$, $B:= f_t^*E_t$, $X:= \left(z-y^2\right)^{-1} $ and $Y:= dy+ \left(\n^{\ker}\right)^2$. Then we know that
\begin{equation*}
\left(z-f_t^* \A_t^2\right)^{-1} - \left(z-\left(y+\n^{\ker}\right)^2\right)^{-1}  = \sum_{n\geq 0} A(BA)^n - X(YX)^n 
\end{equation*}
where the sum is finite. \newline
Let us first look at
\begin{align*}
&P\left(\sum_{n\geq 0} A(BA)^n-X(YX)^n\right)P \\
&= \sum_{n \geq 0} X P(BA)^n P - X(YX)^n \\
&=\sum_{n\geq 0} X P( (PBP+PBQ+QBP +QBQ)A)^n P - X(YX)^n  
\end{align*}
Since $PQ=QP=0$ the only combination in which $QBQ$ can occur is of the following form
\[ PBQ A (QBQA)^k QBP .\]
We know by Lemma \ref{lemma} that 
\[\left\Vert QAQ\right\Vert_{0,2}=\left\Vert \left(z-t\left(D^+\right)^2\left(x,\frac{y}{\sqrt{t}}\right)\right)^{-1}\right\Vert_{0,2}\leq\frac{C}{t}\left(1+\left\vert z\right\vert + \left\vert z \right\vert^2\right)\]
 and again by the definition of $E_t$ that
 \[\left\Vert B\right\Vert_{2,0}=\left\Vert f_t^*E_t\right\Vert_{2,0}\leq C \sqrt{t}.\]
 This proves that
\[ \left\Vert PBQ A (QBQA)^k QBP \right\Vert_{2,0} \leq C t^{-k/2} \left(1+p(\left\vert z\right\vert)\right).\]
By the same argument as above, $PBQ$ and $QBP$ can only occur as
\[ PBQAQBP.\]
Combining these together with inequality (\ref{absch4}) of Lemma \ref{bla} yields to
\begin{align*}
&\left\Vert P\left( \left(z-f_t^* \A_t^2\right)^{-1} - \left(z-\left(y+\n^{\ker}\right)\right)^{-1}\right)P\right\Vert_{0,2}\\
&\leq \left\Vert \sum_{n \geq 0} X P((PBP+PBQ+QBP)A)^nP - X(YX)^n \right\Vert_{0,2} + C t^{-1/2}\left(1+p(\left\vert z\right\vert)\right)\\
&\leq \sum_{n\geq 0} \left\Vert X ((PBP+PBQAQBP)X)^n - X(YX)^n \right\Vert_{0,2} + C t^{-1/2} \left(1+p(\left\vert z\right\vert)\right) \\
&\leq C t^{-1/2}\left(1+p_1\left(\left\vert y\right\vert\right) + p_2\left(\left\vert z\right\vert\right) \right),
\end{align*}
where we used Proposition \ref{prop1} and inequality (\ref{absch4}) of Lemma \ref{bla} in the last step. \newline
For the other estimates we don't need $X(YX)^n$, since $PX(YX)^nP=X(YX)^n$. We know that
\[ A= \begin{pmatrix}
\left(z-y^2\right)^{-1} & 0 \\ 0 & \left(z- t f_t^* \left(D^+\right)^2\right)^{-1}\end{pmatrix} .\]
As before we know by Lemma \ref{bla} that
\[ \left\Vert A \right\Vert_{0,2}\leq C\left(1+\frac{\left\vert z\right\vert}{t}\right) \]
and by Lemma \ref{lemma}
\[ \left\Vert\left(z-tf_t^* \left(D^+\right)^2\right)^{-1}\right\Vert_{0,2}\leq Ct^{-1} \left(1+ \left\vert z \right\vert\right) .\]
In general $\left\Vert B\right\Vert_{2,0}\leq Ct^{1/2}$ but for $PBP$ we even get
\[ \left\Vert PBP\right\Vert_{2,0}\leq C, \]
since the only summand involving $t$ with a positive exponent is
\[ \sqrt{t}f_t^* P\n^{\pi_* V} (D) P= \sqrt{t}f_t^* dy = dy.\]
Now one can easily check inductively that
\begin{align*}
\left\Vert P A(BA)^n Q \right\Vert_{0,2} & \leq C t^{-1/2}\left(1+p\left(\left\vert z\right\vert\right)\right) \\
\left\Vert Q A(BA)^n P \right\Vert_{0,2} & \leq C t^{-1/2}\left(1+p\left(\left\vert z\right\vert\right)\right) \\
\left\Vert Q A(BA)^n Q\right\Vert_{0,2} & \leq C t^{-1} \left(1+p \left(\left\vert z\right\vert\right)\right)
\end{align*}
which proves the other three estimates in the statement.
\end{prf}

\begin{theorem}\label{asymptotik}
There exist constants $C,c>0$ depending on $\ell$, such that for $t$ big enough we get the following estimates. On $B\backslash N_{\eps}$
\begin{equation}
\left\Vert \left.\tr\left(\exp\left(-\A_t^2\right)\right)\right|_{B\backslash N_{\eps}}\right\Vert_{\mathcal{C}^{\ell}(B\backslash N_{\eps})} \leq C e^{-ct}.
\end{equation}
for all $\mathcal{C}^{\ell}$-norms on $\Omega^{\bullet}(B\backslash N_{\eps})$. On $N_{\eps}\cong B_0\times \left(-\varepsilon,\varepsilon\right)$ and for all $\omega\in\Omega^{\bullet}(B)$
\begin{align}
&\left\Vert \left(\int\limits_{-\varepsilon}^{\varepsilon}  \tr\left(\exp\left(-\A_t^2\right)\right)\right)\omega + \sqrt{\pi}\tr\left(\exp\left(-\left(\n^{\ker}\right)^2\right)\right) i^*\omega \right\Vert_{\mathcal{C}^{\ell}(B_0)} \notag\\
& \quad\leq C t^{-1/2} \left\Vert\omega\right\Vert_{\mathcal{C}^{\ell+1}(B)}.
\end{align}
for all $\mathcal{C}^{\ell}$-norms on $\Omega^{\bullet}(B_0)$. If we combine the estimates we have
\begin{align}
&\left\vert \int\limits_B  \tr^{\odd}\left(\exp\left(-\A_t^2\right)\right)\omega+ \sqrt{\pi}\int\limits_{B_0}  \tr\left(\exp\left(-\left(\n^{\ker}\right)^2\right)\right) i^* \omega\right\vert \notag \\
&\quad\leq \frac{C}{\sqrt{t}} \left\Vert\omega\right\Vert_{\mathcal{C}^1(B)}.
\end{align}

\end{theorem}

\begin{prf}
In the following we have constants $C>0$ which may vary from line to line and depend on $\ell$ but not on $t,y,z$ and $x$. \newline
Since $D_b$ is invertible for all $b\in B\backslash N_{\eps}$, we know that 
\[ \left.\left\Vert\tr\left(\exp\left(-\A_t^2\right)\right)\right|_{B\backslash N}\right\Vert_{\mathcal{C}^{\ell}(B\backslash N)}\leq C e^{-ct} \]
on $B\backslash N$ for all $\mathcal{C}^{\ell}$-norms. \newline
On $N$ we know by Proposition \ref{1-P} that
\[ \left\Vert \tr \left( \left( 1-\PP\right)\left( \exp\left( -\A^2_t\right)\right)\right)\right\Vert_{\mathcal{C}^{\ell}(N)} \leq  C f(t) \exp\left(-Kt\right) \]
where $f(t) \in \R  [t, t^{-1}]  $ is a polynomial in $t$ and $t^{-1}$. It remains to show that
\begin{equation}
\left(\int\limits_{-\varepsilon}^{\varepsilon} \tr\left(\PP\left(\exp\left(-\A_t^2\right)\right)\right)\right)\omega+\sqrt{\pi} \tr\left(\exp\left(-\left(\n^{\ker}\right)^2\right) \right)i^*\omega \in\Omega^{\bullet}(B_0)
\end{equation}
is of $O\left(t^{-1/2}\right)$ for all $\mathcal{C}^{\ell}$-norms on $\Omega^{\bullet}(B_0)$. We first prove the statement for $\ell=0$.
\begin{align*}
& \left\Vert \left( \int\limits_{-\varepsilon}^{\varepsilon} \tr\left(\PP\left(\exp\left(-\A_t^2\right)\right)\right)\omega\right) +\sqrt{\pi}\tr\left(\exp\left(-\left(\n^{\ker}\right)^2\right)\right)i^*\omega \right\Vert_{\mathcal{C}^{0}(B_0)}\\
&\leq \left\Vert \int\limits_{-\varepsilon\sqrt{t}}^{\varepsilon\sqrt{t}} \tr\left(\PP\left(\exp\left(-f_t^*\A_t^2\right)\right)f_t^*\omega-\tr\left(\exp\left(-\left(y+\n^{\ker} \right)^2\right)\right)\right) g^*i^*\omega\right\Vert_{\mathcal{C}^0(B_0)} \\
& \quad + C t^{-1/2} e^{-ct}  \\
& \leq \int\limits^{\varepsilon\sqrt{t}}_{-\varepsilon\sqrt{t}} \left(  \left\Vert \tr\left(\PP\left(\exp\left(-f_t^*\A_t^2\right)\right)\right)\right\Vert_{\mathcal{C}^0(B_0)}\left\Vert f_t^* \omega - g^* i^* \omega\right\Vert_{\mathcal{C}^0(B_0)} \right. \\
& \quad \left.+  \left\Vert \tr\left(\exp\left( \PP\left(\exp\left(-f_t^*\A_t^2\right)\right)- \exp\left(-\left(y+\n^{\ker}\right)^2\right)\right) \right)\right\Vert_{\mathcal{C}^0(B_0)}\left\Vert g^* i^* \omega\right\Vert_{\mathcal{C}^0(B_0)}\right) dy \\
&\quad+ Ct^{-1/2} e^{-ct} .
\end{align*}

We write the projection $\PP$ via holomorphic functional calculus. We use the contour $ \Omega_t$ for $\left\vert y\right\vert \leq 1$ and the contour $\Theta_y$ for $1\leq\left\vert y\right\vert \leq \varepsilon \sqrt{t}$. Since $\PP$ projects our operators onto a one-dimensional subspace we make our estimates in the operator instead of the $\left\Vert .\right\Vert_1$-norm.\newline
First case: $\left\vert y \right\vert \leq 1$. 
\begin{align*}
&\left\Vert\tr\left(\PP\left(\exp\left(-f_t^* \A_t^2\right)\right) -\exp\left(-\left(y+\n^{\ker}\right)^2\right)\right)\right\Vert_{\mathcal{C}^{0}(B_0)}\\
& \leq C \left\Vert \frac{1}{2\pi i} \int\limits_{\Omega_t} e^{-z} \left(\left(z-f_t^* \A_t^2\right)^{-1} - \left(z-\left(y+\n^{\ker}\right)\right)^{-1} \right)dz \right\Vert_{0,0} \\
& \leq \frac{C}{2\pi} \int\limits_{\Omega_t} \left\vert e^{-z}\right\vert \left\Vert \left(z-f_t^*\A_t^2 \right)^{-1} - \left(z-\left(y+\n^{\ker}\right)^2\right)^{-1} \right\Vert_{0,0} dz \\
&\leq \frac{C}{2\pi} \int\limits_{\Omega_t} e^{-\RE z} C t^{-1/2} \left(1+ p(\left\vert\RE z\right\vert+1)\right) dz
\end{align*}
here we used Proposition \ref{Normabsch}, $\left\vert y\right\vert \leq 1$ and $\left\vert \IM z\right\vert\leq 1$. Calculating the integral leads to
\begin{equation}\label{ykleiner1}
\left\Vert \tr\left(\PP\left(\exp\left(-f_t^*\A_t^2\right)\right)- \exp\left(-\left(y+\n^{\ker}\right)^2\right)\right)\right\Vert_{\mathcal{C}^0(B_0)} \leq C t^{-1/2}.
\end{equation}
Second case: $1\leq \left\vert y\right\vert\leq \varepsilon\sqrt{t}$. 
\begin{align*}
& \left\Vert \tr\left(\PP\left(\exp\left(-f_t^*\A_t^2\right)\right)- \exp\left(-\left(y+\n^{\ker}\right)^2\right)\right)\right\Vert_{\mathcal{C}^0(B_0)} \\
& \leq \left\Vert \frac{C}{2\pi i} \int\limits_{\Omega_y} e^{-z} \left(\left(z-f_t^*\A_t^2\right)^{-1}-\left(z-\left(y+\n^{\ker}\right)^2\right)^{-1}\right)dz\right\Vert_{0,0} \\
& \leq \frac{C}{2\pi} \int\limits_{\Omega_y} e^{-\RE z} C t^{-1/2} \left(1+ p_1(\left\vert y\right\vert)+p_2(\left\vert\RE z\right\vert +1)\right) dz \\
&\leq C t^{-1/2} e^{-y^2/2} \left(1+p(\left\vert y\right\vert)\right) .
\end{align*}
If we know split the integral over $\left(-\varepsilon\sqrt{t},\varepsilon\sqrt{t}\right)$ into an integral over $\left\vert y\right\vert\leq 1$ and an integral over $1\leq\left\vert y\right\vert\leq\varepsilon\sqrt{t}$ and insert the estimates respectively we obtain
\begin{equation*}
\left\Vert \int\limits^{\varepsilon}_{-\varepsilon}\tr\left(\PP\left(\exp\left(-\A_t^2\right)\right)\right) \omega -\tr\left(\exp\left(-\left(\n^{\ker}\right)^2\right)\right)i^*\omega \right\Vert_{\mathcal{C}^0(B_0)} \leq C t^{-1/2} \left\Vert\omega\right\Vert_{\mathcal{C}^{1}(B)}
\end{equation*}
where we used Lemma \ref{zurueckholen} 
\[ \left\vert \left(f_t^*\omega - g^*i^*\omega\right)\right\vert\leq Ct^{-1/2}\left\Vert \omega\right\Vert_{\mathcal{C}^{1}(B)}.\]

Now we will consider the case $\ell \geq 1$. Let $y_1 ,..., y_{m-1}$ be local coordinates on $B_0$. As already explained in the proof of Proposition \ref{1-P} we can choose any connection $\n$ on $\pi_* V$ to calculate $d\tr\left(\exp\left(-\A_t^2\right)\right)=\tr\left(\n\left(\exp\left(-\A_t^2\right)\right)\right)$. We will consider a diagonal connection $\n$ with respect to the decomposition $L\oplus W$. Then we have the follwing growth in the $\left\Vert \cdot\right\Vert_{2,0}$-norm
\begin{equation*}
\n_{\frac{\del}{\del y_i}} \left(f_t^*\A_t^2\right)=\begin{pmatrix}
 O(1) & O(\sqrt{t}) \\ O(\sqrt{t}) & O(t)
\end{pmatrix}.
\end{equation*}
This uses $\sqrt{t} f_t^* P \n^{\pi_* V}(D)P=\sqrt{t}f_t^* dy=dy$. Using Proposition \ref{Normabsch} we see that in the $\left\Vert\cdot\right\Vert_{0,2}$-norm
\begin{equation*}
\left(z-f_t^*\A_t^2\right)^{-1}=\begin{pmatrix}
O(1) & O\left(\frac{1}{\sqrt{t}}\right) \\
O\left(\frac{1}{\sqrt{t}}\right) & O\left(\frac{1}{t}\right)
\end{pmatrix}.
\end{equation*}
Therefore as in Proposition \ref{1-P} one can see that 
\begin{align*}
\n_{\frac{\del}{\del y_i}}\left(z-f_t^*\A_t^2\right)^{-1}&= \left(z-f_t^*\A_t^2\right)^{-1} \n_{\frac{\del}{\del y_i}}\left(f_t^*\A_t^2\right) \left(z-f_t^*\A_t^2\right)^{-1}\\
&=\begin{pmatrix}
O(1) & O\left(\frac{1}{\sqrt{t}}\right) \\ O\left(\frac{1}{\sqrt{t}}\right) & O\left(\frac{1}{t}\right)
\end{pmatrix}
\end{align*}
in the $\left\Vert \cdot\right\Vert_{0,2}$-norm, just as $\left(z-f_t^*\A_t^2\right)^{-1}$ itself. Inductively this also holds for higher derivatives. Therefore by the same arguments as in the case $\ell =0$ 
\begin{equation*}
\int\limits_{-\eps}^{\eps} \tr\left(\exp\left(-\A_t^2\right)\right) \omega
\end{equation*}
converges with respect to the $\mathcal{C}^{\ell}\left(B_0\right)$-norm for all $\ell\geq 0$.
\end{prf}

\begin{rem}
D. Cibotaru explicitly calculated $\lim_{t\to\infty} \chern(A_t)$ for superconnections $A_t=\n+tA$ on finite rank vector bundles $E\rightarrow B$, see \cite[Theorem 6.7, 6.10]{CibotaruArtikel}. Theorem \ref{asymptotik} can be seen as a generalization to infinite dimensions. In exchange we restrict ourselves to a vector bundle of rank one $\ker D \rightarrow B_0$. In any case the currents we obtain  are not surprising considering what we know from finite dimensions. \newline
The top cohomology class of our representative $-\delta_{B_0} \chern\left(\ker D\rightarrow B_0, \n^{\ker}\right)$ of the analytical index also agrees with the formula given in \cite[Proposition 1.1]{Cibotaru} for $\dim B =3$.
\end{rem}

\begin{prop}\label{thm1}
\[ \beta := \tr^{\ev}\left(\frac{d\A_t}{dt}\exp\left(-\A_t^2\right)\right) dt \in \Omega^{\bullet}(B\times (0,\infty), \CC ) \]
is an integrable differential form.
\end{prop}

\begin{prf}
We know from \cite[Theorem 2.11]{BGS} that $\left\Vert\beta\right\Vert_{\mathcal{C}^{\ell}(B)} \leq C$ for small $t$ and therefore $\tr^{ev}\left(\frac{d\A_t}{dt}\exp\left(-\A_t^2\right)\right) dt$ is integrable as $t\to 0$.\newline
Since $D_b$ is invertible for all $b\in B\backslash N_{\eps}$ we know that $\beta$ is integrable on $B\backslash N_{\eps} \times \left(0,\infty\right)$ \cite[p. 57]{BC}. So let us now consider $\beta$ on $N_{\eps}\cong B_0\times \left(-\eps , \eps\right)$ as $t\to\infty$. Set $S=(1-\delta , 1+\delta)$ and consider the fibre bundle $\widetilde{M}=\left. M\right|_{N_{\eps}} \times S \rightarrow\widetilde{N_{\eps}}=N_{\eps}\times S$ as in the proof of \cite[Theorem 10.32]{BGV}. We denote the extra coordinate in $S$ by $s$ and define the vertical metric by $g_{\widetilde{M}/\widetilde{B}}=s^{-1}g_{M/B}$. The vertical Dirac bundle will be $\widetilde{V}=V\times S\rightarrow\widetilde{M}$, where we take the natural extensions of the given connections. We will write $\sim$ over all induced objects on this family. So let $\widetilde{\A}$ be the Bismut superconnection in this situation which we scale again by the parameter $t\in\left(0, \infty\right)$ as follows
\[ \widetilde{\A}_t= \sqrt{t} \widetilde{D} + \widetilde{\n^{\pi_* V}} - \frac{1}{4 \sqrt{t}} \widetilde{c(T)} .\]
We made assumption \ref{assumption} for the Dirac operators $D$, but
\[ \widetilde{D}_{(b,s)}=\sqrt{s} D_b \]
implies that it also holds for $\widetilde{D}$. We have a bundle $\ker \widetilde{D} \rightarrow \widetilde{B}_0=B_0 \times S$ which is just the pullback of $\ker D\rightarrow B_0$. The submanifold $B_0 \times S$ is of course not compact, but if we allow $\delta$ to become smaller, we get the same uniform estimates as in Theorem \ref{asymptotik}. By combining the estimates (\ref{ykleiner1}) and the following in the proof of Theorem \ref{asymptotik} we see that for $t$ big enough
\begin{equation}\label{abc}
\left\Vert \tr\left(\widetilde{\PP}\left(\exp\left(-f_t^*\widetilde{\A}_t^2\right)\right) - \exp\left(-\left(y+\widetilde{\n}^{\ker}\right)^2\right)\right)\right\Vert_{\mathcal{C}^{\ell}(B_0)} \leq \frac{C}{\sqrt{t}} e^{-y^2/2}.
\end{equation} 
Now we know by \cite[Lemma 10.31]{BGV} or by a straight forward calculation that
\begin{equation}
\left.\tr^{\odd}\left(\exp\left(-\widetilde{\A}_t^2\right)\right)\right|_{s=1} = \tr^{\odd}\left(\exp\left(-\A_t^2\right)\right) - t \tr^{\ev}\left(\frac{d \A_t}{dt}\exp\left(-\A_t^2\right)\right)ds
\end{equation}
and $ \widetilde{\n}^{\ker}$ is just a pullback from $B_0$ and therefore its curvature $\left(\widetilde{\n}^{\ker}\right)^2$ does not involve $ds$. So equation (\ref{abc}) tells us that
\begin{equation}
\left\Vert f_t^* \tr^{ev}\left(\PP\left(\frac{d\A_t}{dt} \exp\left(-\A_t^2\right)\right)\right)\right\Vert_{\mathcal{C}^{\ell}(B_0)} \leq \frac{C}{t^{3/2}} e^{-y^2/2}.
\end{equation}
Using the estimate of Proposition \ref{1-P} for the projection $1-\PP$ we see that
\begin{equation}
\left\Vert f_t^* \trev\left(\frac{d\A_t}{dt}\exp\left(-\A_t^2\right)\right)\right\Vert_{\mathcal{C}^{\ell}(B_0)}\leq \frac{C}{t^{3/2}} e^{-y^2/2}.
\end{equation}
This proves that $f_t^*\trev\left(\frac{d\A_t}{dt}\exp\left(-\A_t^2\right)\right)$ is integrable on $B_0\times \left(-\eps\sqrt{t}, \eps\sqrt{t}\right)\times (0, \infty )$.  By the transformation theorem $\trev\left(\frac{d\A_t}{dt}\exp\left(-\A_t^2\right)\right)$ is integrable on $B_0\times\left(-\eps,\eps\right)\times\left(0,\infty\right)$ and therefore on all of $B\times \left(0,\infty\right)$.
\end{prf}

\begin{defn}
We define 
\begin{equation}
\hat{\eta}:= \frac{1}{\sqrt{\pi}} \int\limits_0^{\infty}\trev\left(\frac{d\A_t}{dt}\exp\left(-\A_t^2\right)\right) dt,
\end{equation}
which is a well-defined differential form on $B$ with coefficients in $L^1(B)$ by Proposition \ref{thm1} and the Fubini theorem. We define $\tilde{\eta}$ by
\begin{equation}
\tilde{\eta}=\sum\limits_k \left(2\pi i\right)^{-k} \hat{\eta}_{[2k]} .
\end{equation}
We can see $\tilde{\eta}$ as a current 
\begin{align*}
\tilde{\eta}\colon  \Omega^{\bullet}(B) & \rightarrow\R , \\
\omega & \mapsto \int\limits_B \tilde{\eta}\wedge \omega
\end{align*}
and define its differential as a current
\begin{equation}
d\tilde{\eta}\left(\omega\right)=-\tilde{\eta}\left(d\omega\right).
\end{equation}
\end{defn}

\begin{rem}
We know even more about the coefficients of $\tilde{\eta}$ than just being integrable. Since we can prove that $\tilde{\eta}$ is smooth outside the tubular neighbourhood $N_{\eps}$ of $B_0$ for all $\eps>0$, it is smooth if restricted to $B\backslash B_0$. But since our estimates where in the $\mathcal{C}^{\ell}$-norm on $B_0$, we also know that $i^*\tilde{\eta}\in \Omega^{\bullet}(B_0)$ is smooth (Dominated convergence Theorem). Therefore the only singularity is at $B_0$ if we cross it in the normal direction.
\end{rem}

\begin{theorem}
We assume that $\V$ admits a spin structure and denote by $\Sigma$ the corresponding spinor bundle. If the Dirac bundle $V$ is of the form $\Sigma\otimes L$ then  
\begin{equation}\label{deta}
d \tilde{\eta} = \Int \AD \chern\left(L, \n^L\right) + \delta_{B_0} \chern \left( \ker D \rightarrow B_0 , \n^{\ker} \right) ,
\end{equation}
where $\delta_{B_0}$ is the current of integration over the hypersurface $B_0$.

\end{theorem}
\begin{proof}
Equation (\ref{deta}) follows from the transgression formula (\ref{tragre})
\begin{equation}
d \int\limits_s^T \tr^{\ev} \left(\frac{d\A_t}{dt}e^{-\A_t^2}\right) = \tr^{\odd} \left(e^{-A_s^2}\right)- \tr^{\odd}\left(e^{-\A_T^2}\right)
\end{equation}
since we know by \cite[Theorem 2.10]{BismutFreed} that for $l=\dim M_b$
\begin{align*}
&\lim_{s\to 0} \frac{1}{\sqrt{\pi}}\tr^{\odd}\left(e^{-\A_s^2}\right)\\
& = \left(2\pi i\right)^{-(l+1)/2}\Int \det\left(\frac{R^{M/B}/2}{\sinh \left(R^{M/B}/2\right)}\right)^{1/2} \tr\left(\exp\left(-\left(\n^L\right)^2\right)\right)
\end{align*}
and by Theorem \ref{asymptotik} that
\begin{equation}
\lim_{T\to\infty}\frac{1}{\sqrt{\pi}}\tr^{\odd}\left(e^{-\A_T^2}\right)=-\delta_{B_0} \tr\left(\exp\left(-\left(\n^{\ker}\right)^2\right)\right).
\end{equation}
If we define the $2\pi i$-scaling as above the resulting formula is
\begin{align*}
d\widetilde{\eta} &=  \Int \det\left(\frac{R^{M/B}/4\pi i}{\sinh\left(R^{M/B}/4\pi i\right)}\right)^{1/2} \tr\left(\exp\left(-\left(\n^L\right)^2/2\pi i\right)\right) \\
& \quad + \delta_{B_0} \tr\left(\exp\left(- \left(\n^{\ker}\right)^2/2\pi i\right)\right) \\
&= \Int \AD \chern\left(L,\n^L\right) + \delta_{B_0} \chern\left(\ker D \rightarrow B_0, \n^{\ker}\right).
\end{align*}
\end{proof}